\documentclass[12pt, reqno]{amsart}

\author[J.~M.~Keith]{Jonathan M. Keith}
\address{School of Mathematics, Monash University, Wellington Rd, Clayton VIC 3800, Australia}
\email{jonathan.keith@monash.edu}

\author[P.~Leonetti]{Paolo Leonetti}
\address{Department of Economics, Universit\`a degli Studi dell'Insubria, via Monte Generoso 71, 21100 Varese, Italy}
\email{leonetti.paolo@gmail.com}



\title{Completeness and additive property for submeasures}

\usepackage[T1]{fontenc}
\usepackage{amsmath}
\usepackage{amssymb}
\usepackage{amsthm}
\usepackage[left=3.5cm, right=3.5cm]{geometry} 
\usepackage{hyperref}
\usepackage{fancyhdr}
\usepackage{enumitem}
\usepackage{comment}
\usepackage{nicefrac}
\usepackage{comment}
\usepackage{bm}
\usepackage{mathrsfs}
\usepackage{graphicx}
\usepackage[utf8]{inputenc}
\usepackage{cancel}
\usepackage{mathtools}

\usepackage{tikz}
\usetikzlibrary{decorations.pathreplacing,matrix,arrows,positioning,automata,shapes,shadows,calc,fadings,decorations,snakes,through,intersections}

\AtBeginDocument{%
   \def\MR#1{}
}

\newtheorem{thm}{Theorem}[section]
\newtheorem{cor}[thm]{Corollary}
\newtheorem{lem}[thm]{Lemma}
\newtheorem{prop}[thm]{Proposition}

\theoremstyle{definition} 
\newtheorem{defi}[thm]{Definition}
\let\olddefi\defi
\renewcommand{\defi}{\olddefi\normalfont}
\newtheorem{question}{Question}
\let\oldquestion\question
\renewcommand{\question}{\oldquestion\normalfont}
\newtheorem{example}[thm]{Example}
\let\oldexample\example
\renewcommand{\example}{\oldexample\normalfont}
\newtheorem{rmk}[thm]{Remark}
\let\oldrmk\rmk
\renewcommand{\rmk}{\oldrmk\normalfont}
\newtheorem{claim}{\textsc{Claim}}

\hypersetup{
    pdftitle={},
    pdfauthor={},
    pdfmenubar=false,
    pdffitwindow=true,
    pdfstartview=FitH,
    colorlinks=true,
    linkcolor=blue,
    citecolor=green,
    urlcolor=cyan
}

\providecommand{\MR}[1]{}

\providecommand{\MR}{\relax\ifhmode\unskip\space\fi MR }

\begin{document}

\maketitle
\thispagestyle{empty}

\begin{abstract} Given an extended real-valued submeasure $\nu$ defined on a field of subsets $\Sigma$ of a given set, we provide necessary and sufficient conditions for which the pseudometric $d_\nu$ defined by $d_{\nu}(A,B):=\min\{1,\nu(A\bigtriangleup B)\}$ for all $A,B \in \Sigma$ is complete. As an application, we show that if $\varphi: \mathcal{P}(\omega)\to [0,\infty]$ is a lower semicontinuous submeasure and $\nu(A):=\lim_n \varphi(A\setminus \{0, 1, \ldots, n-1\})$ for all $A\subseteq \omega$, then $d_\nu$ is complete. This includes the case of all weighted upper densities, fixing a gap in a proof by Just and Krawczyk in [Trans.~Amer.~Math.~Soc.~\textbf{285} (1984), 803--816]. In contrast, we prove that if $\nu$ is the upper Banach density (or an upper density greater than or equal to the latter) then $d_\nu$ is not complete. 
We conclude with several characterizations of completeness in terms of the Stone space of the Boolean algebra $\Sigma/\nu$.
\end{abstract}


\section{Introduction}\label{sec:intro}

Let $\Sigma\subseteq \mathcal{P}(X)$ be a field of subsets of a nonempty set $X$, that is, a nonempty family of subsets of $X$ that is stable under finite unions and complements. Denote also by $\omega$ the set of nonnegative integers and by $\overline{\mathbb{R}}:=[-\infty,\infty]$ the two-point compactification of the reals, endowed with its natural total order. 
A map $\nu: \Sigma \to \overline{\mathbb{R}}$ is said to be a \emph{capacity} if it is monotone (that is, $\nu(A)\le \nu(B)$ for all $A,B \in \Sigma$ with $A\subseteq B$) and satisfies $\nu(\emptyset)=0$. If, in addition, the capacity $\nu$ is also subadditive (that is, $\nu(A\cup B)\le \nu(A)+\nu(B)$ for all $A,B \in \Sigma$), then it is said to be a \emph{submeasure}. Lastly, a capacity (or a submeasure) $\nu$ is called \emph{bounded} if $\nu(X)<\infty$ and \emph{normalized} if $\nu(X)=1$. 
It is clear that submeasures are necessarily nonnegative. 
If there is no risk of confusion, we will write also $\nu: \Sigma\to S$, for some $S\subseteq \overline{\mathbb{R}}$. 
Given a submeasure $\nu: \Sigma \to \overline{\mathbb{R}}$, it is easy to check that the map $d_\nu: \Sigma^2\to \mathbb{R}$ defined by 
$$
\forall A,B \in \Sigma, \quad 
d_\nu(A,B):=\min\{1,\nu(A\bigtriangleup B)\}.
$$
is a pseudometric on the field $\Sigma$. In the context of Boolean algebras, $d_\nu$ is also known as a Fr\'{e}chet--Nikodym metric, see e.g. \cite[Fact 2.3]{MR4550390}. The aim of this work is to provide necessary and sufficient conditions on $\nu$ such that the pseudometric space $(\Sigma, d_\nu)$ is complete, hence to show completeness or non-completeness for large classes of submeasures. 
It is worth noting that the completeness of the pseudometrics $d_\nu$ was important in studies of the famous Maharam's problem \cite{MR18718, MR2456888}, cf. e.g. \cite[p. 62]{MR1487988} and \cite{MR1751539, MR3914716}.

Examples of normalized submeasures include the nonnegative finitely additive probability measures $\mu: \Sigma \to \mathbb{R}$ (for instance, 
the characteristic functions of free ultrafilters on $\omega$ and, of course, all the $\sigma$-additive probability measures), the upper asymptotic density $\mathsf{d}^\star: \mathcal{P}(\omega)\to \mathbb{R}$ defined by 
\begin{equation}\label{eq:defupperasymptotic}
\forall A\subseteq \omega, \quad \mathsf{d}^\star(A):=\limsup_{n\to \infty}\frac{|A\cap n|}{n}
\end{equation}
(where, as usual, we identify each $n \in \omega$ with the set $\{0,1,\ldots,n-1\}$), and all upper densities on $\omega$ considered in \cite{MR3863054, MR4074561, MR4054777}, e.g. the upper analytic, upper P\'{o}lya, and upper Banach
densities. In different contexts, we find also 
exhaustive submeasures 
\cite{MR701524, MR2456888}, 
nonpathological and lower semicontinuous submeasures \cite{MR1711328, MR1708146}, 
and duals of exact games \cite{marinacci2004, MR381720}.

\subsection{Literature results: additive setting} Throughout, we reserve the symbol $\mu$ for finitely additive maps. The following result characterizes the completeness of the pseudometric space $(\Sigma,d_\mu)$ in the case where the domain $\Sigma$ is a $\sigma$-field. 

\begin{thm}\label{thm:firstthmgangopady}
Let $\Sigma \subseteq \mathcal{P}(X)$ be a $\sigma$-field of subsets on a set $X$, and let $\mu: \Sigma \to \mathbb{R}$ be a nonnegative, finitely additive, bounded map. The following are equivalent\textup{:}
    \begin{enumerate}[label={\rm (\roman{*})}]
    \item The pseudometric space $(\Sigma, d_\mu)$ is complete\textup{;}
    \item \label{item:3gango} For every increasing sequence $(A_n: n \in\omega)$ in $\Sigma$, there exists $A \in \Sigma$ such that $\mu(A_n\setminus A)=0$ for all $n \in \omega$ and $\mu(A)=\lim_n \mu(A_n)$\textup{.}
    \end{enumerate}
\end{thm}
\begin{proof}
    See \cite[Theorem 1]{MR1684572} 
    or \cite[Lemma 2.2, Theorem 2.4, and Remark 2.5]{MR1060180}. 
\end{proof}

Additional equivalent conditions can be found in \cite[Theorem 1]{MR1684572}, including completeness of the function space $L_1(\mu)$ or $L_p(\mu)$ for some $p \in [0,\infty)$ (see also \cite[Theorem 3.4]{MR1776031}).  
We refer the reader to the article \cite{MR1776031} of Basile and Bhaskara Rao for a systematic study of completeness of $L_p(\mu)$ spaces and the related spaces $(\Sigma, d_\mu)$. However, we will not consider these characterisations further here. 

The property in item \ref{item:3gango} above is usually called \textbf{\textsc{AP(null)}} and provides a weak form of continuity from below for $\mu$, cf. e.g. \cite{MR1845008}. 
The next result shows that, even in the case $\Sigma=\mathcal{P}(\omega)$, one cannot replace the condition $\mu(A_n\setminus A)=0$ for all $n \in \omega$ in the latter property with $|A_n \setminus A|<\infty$ for all $n \in \omega$. 
To this aim, recall that the asymptotic density $\mathsf{d}(A)$ of a set $A\subseteq \omega$ is defined as in \eqref{eq:defupperasymptotic} by replacing $\limsup_{n\to \infty}$ with $\lim_{n\to \infty}$, provided that latter limit exists (cf. Section \ref{subsec:negative}).
\begin{thm}\label{thm:blass1}
There exists an additive extension $\mu: \mathcal{P}(\omega)\to [0,1]$ of the asymptotic density $\mathsf{d}$ with the following properties: 
\begin{enumerate}[label={\rm (\roman{*})}]
\item \label{item:blass01} The pseudometric space $(\mathcal{P}(\omega), d_\mu)$ is complete\textup{;}
\item \label{item:blass02} There exists an increasing sequence $(A_n: n \in\omega)$ in $\mathcal{P}(\omega)$ such that, if $A \subseteq \omega$ satisfies $|A_n \setminus A|<\infty$ for all $n \in \omega$, then $\mu(A)\neq \lim_n \mu(A_n)$\textup{.}
\end{enumerate}
\end{thm}
\begin{proof}
    It follows by \cite[Corollary 4 or Theorem 5(b)]{MR1845008} and Theorem \ref{thm:firstthmgangopady}.
\end{proof}

In contrast, the same authors also proved the following: 
\begin{thm}\label{thm:blass2}
There exists an additive extension $\mu: \mathcal{P}(\omega)\to [0,1]$ of the asymptotic density $\mathsf{d}$ such that the pseudometric space $(\mathcal{P}(\omega), d_\mu)$ is not complete\textup{.}
\end{thm}
\begin{proof}
    It follows by \cite[Theorem 6]{MR1845008} and the argument following its proof, together with Theorem \ref{thm:firstthmgangopady}. 
\end{proof}

In both results above from \cite{MR1845008}, the authors consider special additive extensions of the asymptotic density $\mathsf{d}$ which are also called \emph{density measures}, namely, maps $\mu^{\mathscr{F}}: \mathcal{P}(\omega)\to [0,1]$ defined by 
\begin{equation}\label{eq:densitymeasure}
\forall A\subseteq \omega, \quad 
\mu^{\mathscr{F}}(A):=\mathscr{F}\text{-}\lim_{n\to \infty} \frac{|A\cap n|}{n}, 
\end{equation}
where $\mathscr{F}$ is a free ultrafilter on $\omega$. Kunisada proved in \cite[Theorem 5.1]{MR3622126} a necessary and sufficient condition for a density measure $\mu^{\mathscr{F}}$ to have  \textbf{\textsc{AP(null)}}.



\subsection{Literature results: non-additive setting} So far, only some rather special cases have been proved in the literature, all in the case $\Sigma=\mathcal{P}(\omega)$. 

The first example is due to Solecki in his seminal work \cite{MR1708146}: a map $\varphi: \mathcal{P}(\omega)\to [0,\infty]$ is said to be a \emph{lower semicontinuous submeasure} (in short, lscsm) if it is a submeasure such that $\varphi(F)<\infty$ for all finite $F \subseteq \omega$ and 
$$
\forall A\subseteq \omega, \quad \varphi(A)=\sup\{\varphi(A\cap n): n \in \omega\}.
$$
Notice that the above property is precisely the lower semicontinuity of the submeasure $\varphi$, regarding its domain $\mathcal{P}(\omega)$ as the Cantor space $2^\omega$, that is, if $A_n \to A$ then $\liminf_n \varphi(A_n) \ge \varphi(A)$. Examples of lscsms include $\varphi(A)=|A|$ or $\varphi(A)=\sum_{n \in A}1/(n+1)$ or $\varphi(A)=\sup_{n\ge 1} |A\cap n|/n$, cf. also \cite[Chapter 1]{MR1711328}. 
\begin{thm}\label{thm:soleckivarphi}
    Let $\varphi: \mathcal{P}(\omega) \to [0,\infty]$ be a lscsm. Then the pseudometric space $(\mathcal{P}(\omega), d_\varphi)$ is complete. 
\end{thm}
\begin{proof}
    See the proof of implication (iii) $\implies$ (i) in \cite[Theorem 3.1]{MR1708146}. 
\end{proof}

Of course, the upper asymptotic density $\mathsf{d}^\star$ defined in \eqref{eq:defupperasymptotic} is not a lscsm (since $\mathsf{d}^\star(A)=0$ for every finite $A\subseteq \omega$). However, 
the first-named author recently proved 
that the analogous result holds for $\mathsf{d}^\star$: 
\begin{thm}\label{thm:jonPAMS}
   The pseudometric space $(\mathcal{P}(\omega), d_{\mathsf{d}^\star})$ is complete. 
\end{thm}
\begin{proof}
    See \cite[Theorem 1.1]{keith2024}.
\end{proof}

Just and Krawczyk claimed in \cite[Lemma 3.1]{MR748847} that the analogue of Theorem \ref{thm:jonPAMS} holds for all weighted upper densities, that is, maps $\nu_f: \mathcal{P}(\omega)\to \mathbb{R}$ defined by $\nu_f(A):=\limsup_n \sum_{i \in A\cap n}f(i)/ \sum_{i \in n}f(i)$, where $f: \omega\to \mathbb{R}$ is an Erd{\H o}s--Ulam function, that is, a nonnegative function such that $\lim_n\sum_{i \in n}f(i)=\infty$ and $\lim_n f(n)/\sum_{i \in n}f(i)=0$. 
It seems to us that their proof contains a gap that does not appear easily fixable.\footnote{More precisely, the last centered upper bound of the distance $\rho_h(a_n/I_h, a/I_h)$ in the proof of \cite[Lemma 3.1]{MR748847} is given by $2^{-n}+\sum_{k=1}^\infty 2^{-(n+k)}$, which goes to $0$ as $n\to\infty$. However, the latter should be replaced by $2^{-n}+k 2^{-n}$, which does not necessarily go to $0$ as $n\to \infty$ because in the construction there is no explicit dependence between $n$ and $k$.}  
A further generalization was claimed by Farah in \cite[Lemma 1.3.3(c)]{MR1711328}, which has been used also in \cite[Proposition 2]{MR1955288}.  
However, as confirmed by Farah in personal communications regarding the latter result, his argument also contains a gap. 
We show in Theorem \ref{thm:alllscsm} below, with a different proof, that their claims were, in fact, correct. 
(A 
revised and corrected version of \cite[Lemma 1.3.3(c)]{MR1711328} can be found in \cite[Proposition 5.2.2]{FarahAQrevisited}.)


\section{Main results}\label{sec:mainresults}

Our first main result characterizes the completeness of the pseudometric spaces $(\Sigma,d_\nu)$, hence providing the analogue of Theorem \ref{thm:firstthmgangopady} for submeasures. 

To this aim, if $\Sigma$ is a field of  subsets on a set $X$, we say then a $\Sigma$-partition $\langle A,A^c\rangle$ of $X$ is an \emph{exact Hahn decomposition} of a map $\lambda: \Sigma\to \overline{\mathbb{R}}$ if $\lambda(B)\ge 0$ for all $B\in \Sigma$ with $B\subseteq A$, and $\lambda(B)\le 0$ for all $B\in \Sigma$ with $B\subseteq A^c$. It is known that exact Hahn decompositions do not necessarily exist, even if $\lambda$ is finitely additive, see \cite[Remark
2.6.3 and Example 11.4.7]{MR751777}. Characterizations of completeness in the finitely additive case through exact Hahn decompositions can be found in \cite[Proposition 3]{MR658058} and \cite[Proposition 6.2]{MR1776031}. It is worth remarking that exact Hahn decompositions played a role to prove a Radon-Nikodym type theorem for certain submeasures, see \cite{MR592152}.

Moreover, given a submeasure $\nu: \Sigma \to \overline{\mathbb{R}}$ and a $d_\nu$-Cauchy sequence $(A_n: n \in \omega)$ with values in $\Sigma$, we write 
$$
\nu^+_{(A_n)}(B):=\lim_{n\to \infty} \nu(A_n \cap B) 
\quad \text{ and }\quad 
\nu^-_{(A_n)}(B):=\lim_{n\to\infty} 
\nu(B\setminus A_n)
$$
for all $B \in \Sigma$. Observe that both limits exist in $[0,\infty]$: in fact, since $(A_n\cap B)\bigtriangleup(A_m\cap B)\subseteq A_n\bigtriangleup A_m$ for all $n,m \in \omega$ then the sequence $(\nu(A_n\cap B): n \in \omega)$ is Cauchy in $[0,\infty]$, hence convergent to $\nu^+_{(A_n)}(B) \in [0,\infty]$ (the case of $\nu^-_{(A_n)}(B)$ is analogous). 
In the following, we assume 
the convention $\infty-\infty:=0$.
\begin{thm}\label{thm:joncharacterization}
    Let $\Sigma \subseteq \mathcal{P}(X)$ be a field of subsets on a set $X$, and let $\nu: \Sigma \to \overline{\mathbb{R}}$ be a submeasure. Then the following are equivalent\textup{:}
    \begin{enumerate}[label={\rm (\roman{*})}]
    \item \label{item:1jon} The pseudometric space $(\Sigma, d_\nu)$ is complete\textup{;}
    \item \label{item:1bjon} For every $d_\nu$-Cauchy sequence $(A_n: n \in\omega)$ in $\Sigma$, the map 
    $\nu^+_{(A_n)}-\nu^-_{(A_n)}$ 
    admits an exact Hahn decomposition\textup{;}
    \item \label{item:2jon} For every increasing sequence $(A_n: n \in\omega)$ in $\Sigma$ such that $\sum_n \nu(A_{n+1}\setminus A_n)<\infty$, there exists $A \in \Sigma$ such that $\nu(A_n\setminus A)=0$ for all $n \in \omega$ and $\lim_n \nu(A\setminus A_n)=0$\textup{.} 
    \end{enumerate}
\end{thm}

Condition \ref{item:2jon} above is the analogue of \textbf{\textsc{AP(null)}} for submeasures. Note that, even if $\nu(F)=0$ for all finite $F\in \Sigma$, the condition 
\textquotedblleft $\nu(A_n\setminus A)=0$ for all $n \in \omega$\textquotedblright\, 
cannot be replaced with 
\textquotedblleft $|A_n\setminus A|<\infty$ for all $n \in \omega$,\textquotedblright\ (see Remark \ref{rmk:nostrongercondition} below). 

As a first immediate application, we show that Gangopadhyay's characterization (namely, Theorem \ref{thm:firstthmgangopady}) holds, more generally, for fields of sets. 
\begin{cor}\label{cor:firstthmgangopadyfields}
Let $\Sigma \subseteq \mathcal{P}(X)$ be a field of subsets on a set $X$, and let $\mu: \Sigma \to \mathbb{R}$ be a nonnegative, finitely additive, bounded map. The following are equivalent\textup{:}
    \begin{enumerate}[label={\rm (\roman{*})}]
    \item \label{item:1gang} The pseudometric space $(\Sigma, d_\mu)$ is complete\textup{;}
    \item \label{item:2gang} For every increasing sequence $(A_n: n \in\omega)$ in $\Sigma$, there exists $A \in \Sigma$ such that $\mu(A_n\setminus A)=0$ for all $n \in \omega$ and $\mu(A)=\lim_n \mu(A_n)$\textup{.}
    \end{enumerate}
\end{cor}
Quite surprisingly, it seems that the above characterization 
cannot be found explicitly in the literature. As remarked by K. P. S. Bhaskara Rao in personal communications, \cite[Theorem 3.3]{MR444898} is a relevant result in the finitely additive setting and may provide an alternative route to a direct proof of Corollary \ref{cor:firstthmgangopadyfields}.

As another immediate application, we recover 
\cite[Proposition 3.4]{keith2024}: 
\begin{cor}\label{cor:sigmasubadditive}
  Let $\Sigma \subseteq \mathcal{P}(X)$ be a $\sigma$-field of subsets on a set $X$, and let $\nu: X\to \overline{\mathbb{R}}$ be a submeasure which is $\sigma$-subadditive. Then 
  $(\Sigma,d_\nu)$ is complete. 
\end{cor}

In Section \ref{subsec:positive}, we provide a large class of submeasures $\nu: \mathcal{P}(\omega)\to [0,\infty]$ such that $(\mathcal{P}(\omega),d_\nu)$ is complete (see Theorem \ref{thm:alllscsm}), which will give a positive proof to the claims in \cite[Lemma 3.1]{MR748847} and \cite[Lemma 1.3.3(c)]{MR1711328}. 
In Section \ref{subsec:negative}, we provide another large class of submeasures $\nu: \mathcal{P}(\omega)\to [0,\infty]$ such that $(\mathcal{P}(\omega),d_\nu)$ is \emph{not} complete (see Theorem \ref{thm:casemustarbiggerupperBanach}). 
Remarkably, this class includes the upper Banach density $\mathsf{bd}^\star$ and the upper Buck density $\mathfrak{b}^\star$  defined in \eqref{eq:defupperBanach} and \eqref{eq:defupperBuck} below, resp.; in both cases, some preliminaries will be necessary. 
Lastly, we will provide in Section \ref{sec:stone} further relationships between the completeness of the pseudometric spaces $(\Sigma, d_\nu)$ and a certain submeasure defined on the power set of the Stone space of the quotient $\Sigma/\nu$. 
Proofs of our results follow in Section \ref{sec:proofs}. 


\subsection{Positive results: ideals and lscsms}\label{subsec:positive} Let $\mathcal{I}$ be an ideal on the nonnegative integers $\omega$, that is, a family of subsets of $\omega$ which is closed under subsets and finite unions. Unless otherwise stated, it is also assumed that $\mathcal{I}$ contains the family $\mathrm{Fin}$ of finite subsets of $\omega$, and that $\omega\notin \mathcal{I}$. Set also $\mathcal{I}^+:=\mathcal{P}(\omega)\setminus \mathcal{I}$.
\begin{prop}\label{prop:positivesimpleexample}
Let $\mathcal{I}$ be an ideal on $\omega$. Then $(\mathcal{P}(\omega), d_{\bm{1}_{\mathcal{I}^+}})$ is complete.
\end{prop}
Regarding ideals as subsets of the Cantor space $2^\omega$, we can speak about their topological
complexity: for instance, $\mathcal{I}$ can be an analytic, Borel, or a $F_\sigma$-subset of $\mathcal{P}(\omega)$. In addition, we say that $\mathcal{I}$ is a $P$-ideal if it is $\sigma$-directed modulo finite sets, that is, for every sequence $(A_n)$ of sets in $\mathcal{I}$ there exists $A \in \mathcal{I}$ such that $A_n\setminus A$ is finite for all $n \in \omega$. For instance, $\mathrm{Fin}$ is a $F_\sigma$ $P$-ideal, and the family of asymptotic density zero sets 
$$
\mathcal{Z}:=\left\{
A\subseteq \omega: \mathsf{d}^\star(A)=0
\right\}
$$
is an analytic $P$-ideal which is not $F_\sigma$. Also, maximal ideals (i.e., the complements of free ultrafilters on $\omega$) do not have the Baire property, hence they are not analytic. We refer to \cite{MR1711328} for an excellent textbook on the theory of ideals. 

Given a lscsm $\varphi: \mathcal{P}(\omega)\to [0,\infty]$, define the family
$$
\mathrm{Exh}(\varphi):=\{S\subseteq \omega: \|S\|_\varphi=0\},
\,\,\,\, \text{ where }\,\,\|S\|_\varphi:=\lim_{n\to \infty} \varphi(S\setminus n).
$$
Informally, $\|S\|_\varphi$ stands for the $\varphi$-mass at infinity of the set $S$. 
A classical result of Solecki \cite[Theorem 3.1]{MR1708146} states that an ideal $\mathcal{I}$ on $\omega$ is an analytic $P$-ideal if and only if there exists a lscsm $\varphi$ such that 
$$
\mathcal{I}=\mathrm{Exh}(\varphi)
\quad \text{ and }\quad 
\varphi(\omega)<\infty. 
$$

We remark that the family of analytic $P$-ideals is large and includes, among others, all Erd{\H o}s--Ulam ideals introduced by Just and Krawczyk in \cite{MR748847}, ideals generated by nonnegative regular matrices \cite{Filipow18, MR4041540}, 
the Fubini products $\emptyset \times \mathrm{Fin}$, which can be defined as 
$
\{A\subseteq \omega: \forall n \in \omega, A \cap I_n \in \mathrm{Fin}\},
$ 
where $(I_n)$ is a given partition of $\omega$ into infinite sets,  
certain ideals used by Louveau and Veli\u{c}kovi\'{c} \cite{Louveau1994}, 
and, more generally, density-like ideals and generalized density ideals \cite{MR3436368, MR4404626}. Additional pathological examples can be found in \cite{MR0593624}. 
It has been suggested in \cite{MR4124855, MR3436368} that the theory of analytic $P$-ideals may have some relevant yet unexploited potential for the study of the geometry of Banach spaces.

It is easy to check that, for each finite lscsm $\varphi$, the map $\|\cdot\|_\varphi$ is monotone, subadditive, and invariant under finite modifications, cf. \cite[Lemma 1.3.3(a)-(b)]{MR1711328}. 
In addition, the lscsm $\varphi$ in such a representation is not necessarily unique, cf. Example \ref{example:Zrepresentation} below for the case of the ideal $\mathcal{Z}$. On a similar note, we will show in Proposition \ref{rmk:topequivalent} that, if $\varphi_1, \varphi_2$ are two lscsms with $\mathrm{Exh}(\varphi_1)=\mathrm{Exh}(\varphi_2)$, then the pseudometrics $d_{\|\cdot\|_{\varphi_1}}$ and $d_{\|\cdot\|_{\varphi_2}}$ are topologically equivalent, though not necessarily metrically equivalent (see Proposition \ref{prop:notmetricalequivalence}). 

With the above context, our main positive result follows: 
\begin{thm}\label{thm:alllscsm}
    Let $\varphi: \mathcal{P}(\omega) \to [0,\infty]$ be a lscsm. Then $(\mathcal{P}(\omega), d_{\|\cdot\|_\varphi})$ is complete. 
\end{thm}

Since $\mathcal{Z}$ is an analytic $P$-ideal,  Theorem \ref{thm:alllscsm} provides a generalization of Theorem \ref{thm:jonPAMS}, cf. also Example \ref{example:Zrepresentation} below. 
In addition, a special case of Theorem \ref{thm:alllscsm} follows by a result of Solecki in \cite[Theorem 3.4]{MR1708146}, where he shows that the topology induced by $d_{\|\cdot\|_\varphi}$ is discrete if and only if $\mathrm{Exh}(\varphi)$ is a $F_\sigma$-ideal. It seems plausible that an alternative proof of Theorem \ref{thm:alllscsm} could be obtained using the notion of algebraic convergence studied in \cite[Section 2]{MR1487988}, cf. also \cite{MR1751539, 
MR3914716}.

As remarked by Farah in personal communications, an instance of Theorem \ref{thm:alllscsm} 
can also be obtained as follows: suppose that $\varphi$ is a lscsm such that $\mathrm{Exh}(\varphi)$ is a generalized density ideal. Then the Boolean algebra $\mathcal{P}(\omega)/\mathrm{Exh}(\varphi)$ endowed with the metric induced by $d_{\|\cdot\|_\varphi}$ is countably saturated, see \cite[Proposition 2.8]{MR3427592} and cf. also \cite[Theorem 16.5.1]{MR3971570} and the argument preceding \cite[Corollary 6.5]{Farah24}; thus, completeness is a consequence of the latter property. However, it is an open question whether the analogous argument holds for all analytic $P$-ideals; see the last paragraph of \cite[Section 2]{MR3427592}.

The density measures $\mu^{\mathscr{F}}$ in \eqref{eq:densitymeasure} 
used by Blass et al. \cite{MR1845008} do not fall into the realm of Theorem \ref{thm:alllscsm}: in fact, for each free ultrafilter $\mathscr{F}$ on $\omega$, the family $\{A\subseteq \omega: \mu^{\mathscr{F}}(A)=0\}$ is an ideal which is not analytic by the proof of \cite[Theorem 2.10]{MR4052262}, hence it cannot coincide with $\mathrm{Exh}(\varphi)$ for any lscsm $\varphi$.


\subsection{Negative results: upper densities}\label{subsec:negative} 
In this section, complementing Proposition \ref{prop:positivesimpleexample}, we provide a simple class of pseudometric spaces $(\mathcal{P}(\omega), d_\nu)$ which are \emph{not} complete.
\begin{prop}\label{prop:negativesimpleexample}
Let $\mathcal{I}$ be an ideal on $\omega$. Let also $(a_n: n \in \omega)$ be a sequence of strictly positive reals such that $\sum_n a_n=1$. Define the submeasure $\nu$ by  
$$
\forall A\subseteq \omega, \quad 
\nu(A):=\bm{1}_{\mathcal{I}^+}(A)+\sum_{n \in A}a_n.
$$
Then $(\mathcal{P}(\omega), d_{\nu})$ is not complete.
\end{prop}

Then, we study natural instances of $\nu$ such as upper Banach density. To this aim, following 
\cite
[Definition 1]
{MR4054777}, 
we say that a map $\mu^\star: \mathcal{P}(\omega)\to \mathbb{R}$ is an \emph{upper density} on $\omega$ if 
it satisfies the following properties: 
\begin{enumerate}[label={\rm (\textsc{f}\arabic{*})}]
\item $\mu^\star(\omega)=1$;
\item $\mu^\star(A)\le \mu^\star(B)$ for all $A\subseteq B\subseteq \omega$;
\item $\mu^\star(A\cup B)\le \mu^\star(A)+\mu^\star(B)$ for all $A,B\subseteq \omega$; 
\item $\mu^\star(k\cdot A)=\mu^\star(A)/k$ for all $A\subseteq \omega$ and all nonzero $k \in \omega$;
\item $\mu^\star(A+h)=\mu^\star(A)$ for all $A\subseteq \omega$ and $h \in \omega$.
\end{enumerate}
It is readily seen that, if $\mu^\star$ is an upper density, then $0\le \mu^\star(A)\le 1$ for all $A\subseteq \omega$, and $\mu^\star(F)=0$ for every finite $F\subseteq \omega$, see \cite[Proposition 2 and Proposition 6]{MR4054777}; in particular, each $\mu^\star$ is a normalized submeasure. 

Examples of upper densities include the upper asymptotic density $\mathsf{d}^\star$ defined in \eqref{eq:defupperasymptotic}, the upper $\alpha$-densities with $\alpha \ge -1$ \cite[Example 4]{MR4054777}, the upper Banach density $\mathsf{bd}^\star$ defined by 
\begin{equation}\label{eq:defupperBanach}
\forall A \subseteq \omega, \quad
\mathsf{bd}^\star(A):=\lim_{n \to \infty}\, \max_{k \in \omega}\,\frac{|A \cap [k,k+n)|}{n},
\end{equation}
the upper analytic density \cite[Example 6]{MR4054777}, the upper Polya density \cite[Example 8]{MR4054777}, and the upper Buck density $\mathfrak{b}^\star$ defined by 
\begin{equation}\label{eq:defupperBuck}
\forall A \subseteq \omega, \quad
\mathfrak{b}^\star(A):=\inf_{X\in \mathscr{A}, A\subseteq X}\,\mathsf{d}^\star(A),
\end{equation}
where $\mathscr{A}$ stands for the family of finite unions of infinite arithmetic progressions $k\cdot \omega+h$ (with $k,h \in \omega$ and $k\neq 0$). With these premises, recall that 
\begin{equation}\label{eq:inequalitiesupperdensities}
\mathsf{d}^\star(A)\le \mathsf{bd}^\star(A)
\quad \text{ and }\quad 
\mu^\star(A)\le \mathfrak{b}^\star(A)
\end{equation}
for all $A\subseteq \omega$ and all upper densities $\mu^\star$ on $\omega$, see \cite[Theorem 3]{MR4054777}. 

For any upper density $\mu^\star$ on $\omega$, we denote by $\mu_\star$ its lower dual, that is, the map defined by 
$
\mu_\star(A):=1-\mu^\star(\omega\setminus A)
$
for all $A\subseteq \omega$. We denote by $\mu$ the density induced by the pair $(\mu^\star, \mu_\star)$, that is, the restriction of $\mu^\star$ to the family 
$
\mathrm{dom}(\mu):=\{A\subseteq \omega: \mu^\star(A)=\mu_\star(A)\}. 
$
For instance, if $\mu^\star$ is the upper asymptotic density $\mathsf{d}^\star$, then $\mathsf{d}_\star$ is the lower asymptotic density, $\mathsf{d}$ is the asymptotic density, and $\mathrm{dom}(\mathsf{d})$ is the family of sets $A\subseteq \omega$ which admits asymptotic density (that is, $\mathsf{d}^\star(A)=\mathsf{d}_\star(A)$). 

Our main negative result follows: 

\begin{thm}\label{thm:casemustarbiggerupperBanach}
   Let $\mu^\star$ be an upper density on $\omega$ such that $\mathsf{bd}^\star(A)\le \mu^\star(A)$ for all $A\subseteq \omega$. Then $(\mathcal{P}(\omega), d_{\mu^\star})$ is not complete.    
\end{thm}

Thanks to \eqref{eq:inequalitiesupperdensities}, the above result applies to the upper Banach density $\mathsf{bd}^\star$ and the upper Buck density $\mathfrak{b}^\star$. In addition, thanks to \cite[Proposition 10]{MR4054777}, it applies also to the upper density $(\alpha (\textsf{bd}^\star)^q + (1-\alpha) (\mathfrak{b}^\star)^q)^{1/q}$ for every $q \in [1,\infty)$ and every $\alpha \in [0,1]$ (more generally, the set of upper densities satisfying the hypothesis of Theorem \ref{thm:casemustarbiggerupperBanach} is countably $q$-convex for every $q \in [1,\infty))$. 


\subsection{Relationships with Stone spaces}\label{sec:stone} 
Let $\Sigma$ and $\nu$ be as in the statement of Theorem \ref{thm:joncharacterization}, and let $\nu^\star: \mathcal{P}(X) \to \overline{\mathbb{R}}$ be the map defined by 
$$
\forall A\subseteq X, \quad 
\nu^\star(A):=\inf\left\{\nu(C): C \in \Sigma \text{ and }A\subseteq C\right\}. 
$$
Of course, $\nu^\star$ is a submeasure which extends $\nu$. We write 
$
\overline{\Sigma}
$  
for the closure of $\Sigma$ in the space $(\mathcal{P}(X), d_{\nu^\star})$. As it follows by Remark \ref{rmk:peanocompleation}, this is precisely the classical Peano--Jordan completion in the finitely additive case. 
With a small abuse of notation, given a field  $\Sigma^\prime\subseteq \mathcal{P}(X)$, we write $(\Sigma^\prime, d_{\nu^\star})$ for the pseudometric space $(\Sigma^\prime, d_{\psi})$, where $\psi$ stands for the restriction of $\nu^\star$ on $\Sigma^\prime$. 

\begin{lem}\label{lem:easyimplication}
    Let $\Sigma \subseteq \mathcal{P}(X)$ be a field of subsets on a set $X$, and let $\nu: \Sigma \to \overline{\mathbb{R}}$ be a submeasure. 
    Then $(\overline{\Sigma}, d_{\nu^\star})$ is complete whenever $(\Sigma, d_\nu)$ is complete\textup{.}
\end{lem}

It is worth noting that, as it follows by \cite[Proposition 5.1 and Example 5.2]{MR1776031}, the converse implication of Lemma \ref{lem:easyimplication} fails even if $\nu$ is finitely additive.

Several equivalent conditions for the completeness of $(\overline{\Sigma}, d_{\nu^\star})$ will be given in Theorem \ref{thm:stonecharacterization} below. To this aim, denote by $\Sigma/\nu$ the quotient space of $\Sigma$ with the ideal $
\{A \in \Sigma: \nu(A)=0\}$. Equip $\Sigma/\nu$ with the quotient topology, and denote the canonical projection by 
$$
\pi: \Sigma\to \Sigma/\nu. 
$$
Thus, a generic element of $\Sigma/\nu$ is an equivalence class of the type $\pi(A)=\{C \in \Sigma: \nu(A \bigtriangleup C)=0\}$ for some set $A \in \Sigma$. 
%
Now, let $S$ be the Stone space of the Boolean algebra $\Sigma/\nu$, that is, the set of ultrafilters on $\Sigma/\nu$, equipped with its usual topology. 
Recall that $S$ is compact and totally disconnected. 
Denote by $\mathcal{C}$ the field of all clopen (that is, simultaneously open and closed) subsets of $S$, so that $\mathcal{C}\subseteq \mathcal{P}(S)$. 
We denote by $\sigma(\mathcal{C})$ the $\sigma$-field generated by $\mathcal{C}$. 
It is known that $\sigma(\mathcal{C})$ coincides with the Baire $\sigma$-field of $S$, that is, the smallest $\sigma$-field containing the compact $G_\delta$ sets (the short argument is essentially contained in Lemma~\ref{lem:tildenulemma}\ref{item:6tildenu} below.   
By Stone's representation theorem, there exists a Boolean isomorphism 
$$
\phi: \Sigma/\nu \to \mathcal{C}, 
$$ 
see e.g. \cite[Chapter 1, Section 8]{MR242724} for a textbook exposition on Boolean algebras. 
For each $B\subseteq S$ let $\Delta(B)$ be the family of all sequences $(A_n: n \in \omega)$ with values in $\Sigma$ such that $B\subseteq \bigcup_n \phi(\pi(A_n))$ (of course, each $\Delta(B)$ is nonempty since $\phi(\pi(X))=S$).  
Thus, define the outer measure $\hat{\nu}: \mathcal{P}(S) \to \overline{\mathbb{R}}$ by 
$$
\forall B\subseteq S, \quad 
\hat{\nu}(B):=\inf\left\{ \sum_{n\in \omega}\nu(A_n): (A_n) \in \Delta(B)\right\},
$$
and observe that $\hat{\nu}$ is $\sigma$-subadditive by the very same argument used in the first part of the proof of \cite[Lemma III.5.5]{MR1009162}. Define also the map $\tilde{\nu}: \mathcal{P}(S) \to \overline{\mathbb{R}}$ by 
$$
\forall B\subseteq S, \quad 
\tilde{\nu}(B):=\inf\left\{\nu(A): A \in \Sigma \text{ and }B\subseteq \phi(\pi(A))\right\}. 
$$
Of course, both $\tilde{\nu}$ and $\tilde{\nu}$ are submeasures, and $\hat{\nu}\le \tilde{\nu}$. 
Let also $\partial B$ be the set of boundary points of $B\subseteq S$. Finally, let $\mathrm{Cl}_{\hat{\nu}}(\mathcal{C})$ be the closure of $\mathcal{C}$ with respect to  $d_{\hat{\nu}}$, and $\mathcal{N}_{\hat{\nu}}$ denote the family of $\hat{\nu}$-null sets. 
Several properties of $\mathrm{Cl}_{\hat{\nu}}(\mathcal{C})$ can be found in
Lemma \ref{lem:closureclopenlemma} and the subsequent results.  
Analogous notations are used for $\tilde{\nu}$. Lastly, let $\mathrm{nwd}(S)$ denote the family of the nowhere dense sets in $S$.

With the above premises, we have the following characterization, in the same spirit of \cite[Theorem 4.2]{MR1776031} for the finitely additive case, cf. also \cite{MR1601849, MR279579} and the related comments in \cite[Section 4]{MR1776031}.

\begin{thm}\label{thm:stonecharacterization}
   Let $\Sigma \subseteq \mathcal{P}(X)$ be a field of subsets on a set $X$, and let $\nu: \Sigma \to \overline{\mathbb{R}}$ be a submeasure. 
    Then the following are equivalent\textup{:}

     
    \begin{enumerate}[label={\rm (\textsc{a}\arabic*)}, itemsep=1mm]
    \item \label{item:B1stone} The pseudometric space $(\overline{\Sigma}, d_{\nu^\star})$ is complete\textup{;}

    \item \label{item:B2stone} For every increasing sequence $(A_n: n \in\omega)$ in $\Sigma$ such that $\sum_n \nu(A_{n+1}\setminus A_n)<\infty$, there exists a decreasing sequence $(D_m: m \in \omega)$ in $\Sigma$ such that $\nu(A_n\setminus D_m)=0$ for all $n,m \in \omega$ and $\lim_n \nu(D_n\setminus A_n)=0$\textup{;}

    \item \label{item:B3stone} For every increasing sequence $(A_n: n \in\omega)$ in $\Sigma$ such that $\sum_n \nu(A_{n+1}\setminus A_n)<\infty$, there exists $C \subseteq X$ such that $\nu^\star(A_n\setminus C)=0$ for all $n \in \omega$ and $\lim_n \nu^\star(C\setminus A_n)=0$\textup{;}
    
    \end{enumerate}


    \begin{enumerate}[label={\rm (\textsc{b}\arabic*)}, itemsep=1mm]

     \item \label{item:C1stone} $\hat{\nu}(U)=\hat{\nu}(\overline{U})$ for every set $U\subseteq S$\textup{;}



     \item \label{item:C1ABCD} $\hat{\nu}(U)=\hat{\nu}(\overline{U})$ for every set $U\subseteq S$ which is a countable union of clopen sets\textup{;}

     \item \label{item:C2stone} For every increasing sequence $(A_n: n \in\omega)$ in $\Sigma$ such that $\sum_n \nu(A_{n+1}\setminus A_n)<\infty$, there exists a decreasing sequence $(D_m: m \in \omega)$ in $\Sigma$ such that $\nu(A_n\setminus D_m)=0$ for all $n,m \in \omega$ and $\sup_n \nu(A_n)=\inf_m \nu(D_m)$\textup{;}

     \item \label{item:C3stone} For every increasing sequence $(A_n: n \in\omega)$ in $\Sigma$ such that $\sum_n \nu(A_{n+1}\setminus A_n)<\infty$, there exists $C \subseteq X$ such that $\nu^\star(A_n\setminus C)=0$ for all $n \in \omega$ and $\sup_n \nu(A_n)=\nu^\star(C)$\textup{;}

      \end{enumerate}


      \begin{enumerate}[label={\rm (\textsc{c}\arabic*)}, itemsep=1mm]

     \item \label{item:D1stone} The pseudometric space $(\mathcal{P}(X), d_{\nu^\star})$ is complete\textup{;}

     \item \label{item:D2stone} The pseudometric space $(\Sigma^\prime, d_{\nu^\star})$ is complete for every closed field  $\Sigma^\prime\subseteq \mathcal{P}(X)$ which contains $\Sigma$\textup{;}

     \item \label{item:D3stone} The pseudometric space $(\Sigma^\prime, d_{\nu^\star})$ is complete for some closed field  $\Sigma^\prime\subseteq \mathcal{P}(X)$ which contains $\Sigma$\textup{;}

      \end{enumerate}


      \begin{enumerate}[label={\rm (\textsc{d}\arabic*)}, itemsep=1mm]

     \item \label{item:DDD1stone} $\hat{\nu}=\tilde{\nu}$\textup{;}

     \item \label{item:DDD2stone} $\tilde{\nu}$ is $\sigma$-subadditive\textup{;}

     \item \label{item:DDD3stone} For every $B \in \mathrm{Cl}_{\hat{\nu}}(\mathcal{C})$ and $\varepsilon>0$, there exist $A,C \in \mathcal{C}$ such that $A\subseteq B\subseteq C$ and $\tilde{\nu}(C\setminus A)<\varepsilon$\textup{;}

     \item \label{item:DDD3Bstone} For every $B \in \mathrm{Cl}_{\hat{\nu}}(\mathcal{C})$, there exist sets $A,C \subseteq S$ such that $A$ is open, $C$ is closed, $A\subseteq B\subseteq C$, and $\hat{\nu}(C\setminus A)=0$\textup{;}

     \item \label{item:DDD4stone} $\mathrm{Cl}_{\hat{\nu}}(\mathcal{C})=\mathrm{Cl}_{\tilde{\nu}}(\mathcal{C})$\textup{;}


    \item \label{item:DDD7stone} A subset of $S$ is the boundary of some element of $\mathrm{Cl}_{\hat{\nu}}(\mathcal{C})$ if and only if it is closed and $\hat{\nu}$-null\textup{;}

    \item \label{item:DDD9stone} $\mathcal{N}_{\hat{\nu}} = \mathrm{nwd}(S) \cap \mathrm{Cl}_{\hat{\nu}}(\mathcal{C})$ and $\partial B \in \mathrm{Cl}_{\hat{\nu}}(\mathcal{C})$ for every $B \in \mathrm{Cl}_{\hat{\nu}}(\mathcal{C})$\textup{;}
      \end{enumerate}


      \begin{enumerate}[label={\rm (\textsc{e}\arabic*)}, itemsep=1mm]

     \item \label{item:E1stone} $\tilde{\nu}(\partial B)=0$ for every set $B \in \mathrm{Cl}_{\hat{\nu}}(\mathcal{C})$\textup{;}

     \item \label{item:E2stone} 
     $\hat{\nu}(\partial B)=0$ for every set $B \in \mathrm{Cl}_{\hat{\nu}}(\mathcal{C})$\textup{;}

     \item \label{item:E3stone} 
     $\hat{\nu}(\partial B)=0$ 
     for every set $B \in \mathrm{Cl}_{\hat{\nu}}(\mathcal{C})$ which is a countable union of clopen sets\textup{.}

     \end{enumerate}

\end{thm}

As an application in the case where $\Sigma$ is closed, we provide some additional equivalences for the completeness of $(\Sigma, d_\nu)$ in terms of properties of the Stone space $S$.

\begin{cor}\label{cor:corollaryjoncharacterization}
    Let $\Sigma \subseteq \mathcal{P}(X)$ be a field of subsets on a set $X$ such that $\Sigma=\overline{\Sigma}$, and let $\nu: \Sigma \to \overline{\mathbb{R}}$ be a submeasure. Then the following are equivalent\textup{:}
    \begin{enumerate}[label={\rm (\roman{*})}]
    \item \label{item:cor1jon} The pseudometric space $(\Sigma, d_\nu)$ is complete\textup{;}
    \item \label{item:cor2jon} For every $B \in \mathrm{Cl}_{\hat{\nu}}(\mathcal{C})$ there exists $A \in \Sigma$ such that $\tilde{\nu}(\phi(\pi(A))\bigtriangleup B)=0$\textup{;}
    \item \label{item:cor2Bjon} For every $B \in \mathrm{Cl}_{\hat{\nu}}(\mathcal{C})$ there exists $A \in \Sigma$ such that $\hat{\nu}(\phi(\pi(A))\bigtriangleup B)=0$\textup{;}
    \item \label{item:cor3jon} For every 
    $B\in \mathrm{Cl}_{\hat{\nu}}(\mathcal{C})$ 
    which is a countable union of clopen sets, 
    there exists $A \in \Sigma$ such that $\hat{\nu}(\phi(\pi(A))\bigtriangleup B)=0$\textup{.}
    \end{enumerate}
\end{cor}

To conclude, we remark that, in the case where $\nu$ is finitely additive, an equivalence in the same spirit of \ref{item:B1stone} $\Longleftrightarrow$ \ref{item:C1stone} can be obtained as a consequence of \cite[Theorem 4.2 and Proposition 5.1]{MR1776031} (however, in the latter work the definition of $\hat{\nu}$ is different, see \cite[Section 1.15]{MR1776031} for details). 





\section{Proofs of the main results in Section \ref{sec:mainresults}}\label{sec:proofs}

Let us start with the proof of our main characterization.

\begin{proof}
[Proof of Theorem \ref{thm:joncharacterization}] 
\ref{item:1jon} $\implies$ \ref{item:1bjon}. 
Let $(A_n: n \in \omega)$ be a $d_\nu$-Cauchy sequence in $\Sigma$ and let $A\in\Sigma$ be a $d_\nu$-limit. 
Then, for every $B\in\Sigma$, since $(A_n\cap B)\triangle(A\cap B)\subseteq A_n\triangle A$ for all $n\in \omega$,  
we have
$$
\nu^+_{(A_n)}(B)=\lim_{n\to \infty} \nu(A_n\cap B)=\nu(A\cap B),
$$
and similarly $\nu^-_{(A_n)}(B)=\nu(B\setminus A)$. 
Now, define $\lambda:=\nu^+_{(A_n)}-\nu^-_{(A_n)}$. 
It follows that $\lambda_{(A_n)}(B)\ge 0$ for all $B \in \Sigma$ with $B\subseteq A$, and $\lambda_{(A_n)}(B)\le 0$ for all $B \in \Sigma$ with $B\subseteq A^c$. Hence $\langle A,A^c\rangle$ is an exact Hahn decomposition of $\lambda_{(A_n)}$.

\medskip

\ref{item:1bjon} $\implies$ \ref{item:2jon}. 
Pick an increasing sequence $(A_n: n \in\omega)$ in $\Sigma$ such that $\sum_n \nu(A_{n+1}\setminus A_n)<\infty$. Observe that, for all $n,m \in \omega$ with $m\ge n$, we have $d_\nu(A_n,A_m)\le \nu(A_m\setminus A_n)\le \sum_{k\ge n}\nu(A_{k+1}\setminus A_k)$, hence $(A_n: n \in\omega)$ is $d_\nu$-Cauchy. It follows by item \ref{item:1bjon} that 
the map $\lambda:=\nu^+_{(A_n)}-\nu^-_{(A_n)}$ 
admits an exact Hahn decomposition $\langle A,A^c\rangle$, for some $A \in \Sigma$. We claim that $A$ 
is a set witnessing \ref{item:2jon}. 

To this aim, fix $k \in \omega$. Since $A_k\setminus A\subseteq A^c$, then $\lambda(A_k\setminus A)\le 0$. But $\nu^-_{(A_n)}(A_k\setminus A)=\lim_n \nu((A_k\setminus A)\setminus A_n)=0$. It follows that $\nu^+_{(A_n)}(A_k\setminus A)=\lambda(A_k\setminus A)\le 0$, hence $0=\nu^+_{(A_n)}(A_k\setminus A)=\lim_n \nu(A_n \cap (A_k\setminus A))=\nu(A_k\setminus A)$.

Similarly, $A\setminus A_k\subseteq A$, hence $\lambda(A\setminus A_k)\ge 0$, so $\nu^+_{(A_n)}(A\setminus A_k)\ge \nu^-_{(A_n)}(A\setminus A_k)$. 
Since $(A_n: n \in \omega)$ is increasing, we have also 
$$
\nu^-_{(A_n)}(A\setminus A_k)=\lim_{n\to\infty}\nu((A\setminus A_k)\setminus A_n)=\lim_{n\to\infty}\nu(A\setminus A_n).
$$
Putting everything together, for all $k \in \omega$ it follows that 
\begin{displaymath}
    \begin{split}
        \lim_{n\to\infty}\nu(A\setminus A_n)
        &=\nu^-_{(A_n)}(A\setminus A_k)
        \le \nu^+_{(A_n)}(A\setminus A_k)\\
        &
        =\lim_{n\to \infty}\nu(A_n \cap (A\setminus A_k))
        \le \lim_{n\to \infty}\nu(A_n \setminus A_k) 
        \\
        &\le \lim_{n\to \infty}\sum_{i=k}^{n-1} \nu(A_{i+1}\setminus A_i) 
        = \sum_{i\ge k}\nu(A_{i+1}\setminus A_i).
    \end{split}
\end{displaymath}
By the arbitrariness of $k$, we conclude that $\lim_{n}\nu(A\setminus A_n)=0$.



\medskip

\ref{item:2jon} $\implies$ \ref{item:1jon}. Let $(A_n: n \in \omega)$ be a Cauchy sequence in the pseudometric space $(\Sigma,d_\nu)$. Without loss of generality, we can assume that $\nu(A_i \triangle A_j) < 2^{-i}$ for all $i,j \in \omega$ with $i \le j$, since there is a subsequence with this property and a Cauchy sequence converges to the same limits as any subsequence. For all $i,j \in \omega$ with $i \le j$, define the $\Sigma$-measurable set 
$$
B_{i,j}:= \bigcap_{k=i}^j A_k 
$$
and note that 
\begin{equation}\label{eq:firstinequality}
\nu(A_i \triangle B_{i,j}) \leq \sum_{k=i}^{j-1} \nu(A_k \setminus A_{k+1}) < 2^{1-i}.
\end{equation}
In addition, for each $i\in \omega$, $(B_{i,j}^c: j \ge i)$ is an increasing sequence with
$$
\sum_{j \ge i} \nu(B_{i,j+1}^c \setminus B_{i,j}^c) \leq 
\sum_{j\ge i} \nu(A_j \setminus A_{j+1}) 
< 2^{1-i}.
$$
Hence, by assumption, for each $i \in \omega$ there exists $B_i \in \Sigma$ such that $\nu(B_i \setminus B_{i,j}) = \nu(B_{i,j}^c \setminus B_i^c) = 0$ for all $j \geq i$, and $\lim_j \nu(B_{i,j} \setminus B_i) = \lim_j \nu(B_i^c \setminus B_{i,j}^c)= 0$. Thus 
\begin{equation}\label{eq:firstinequality2}
\lim_{j\to \infty}\nu(B_{i,j} \bigtriangleup B_i) =0.
\end{equation}

At this point, for all $i \in \omega$, define the $\Sigma$-measurable set
$$
C_i:=\bigcup_{k=0}^i B_k.
$$
Observe that, for all integers $i,k\in \omega$ with $k\ge i$, we have
\begin{displaymath}
    \begin{split}
\nu(B_i \triangle C_i) 
&\leq \sum_{j=0}^{i-1} \nu(B_j \setminus B_{j+1}) \\
&\leq \sum_{j=0}^{i-1} (\nu(B_j \setminus B_{j,k}) + \nu(B_{j,k} \setminus B_{j+1,k}) + \nu(B_{j+1,k} \setminus B_{j+1})) \\
&= \sum_{j=0}^{i-1} \nu(B_{j+1,k} \setminus B_{j+1}).
    \end{split}
\end{displaymath}
By the hypothesis as applied to each $B_j$ and letting $k\to \infty$, this implies that 
\begin{equation}\label{eq:firstinequality3}
\forall i \in \omega, \quad 
\nu (B_i \bigtriangleup C_i) = 0.
\end{equation}
Moreover, $(C_i: i \in \omega)$ is an increasing sequence and, for all $i,j \in \omega$ with $i\le j$,
\begin{eqnarray*}
C_{i+1} \setminus C_i &\subseteq& B_{i + 1} \setminus B_i \\
&\subseteq& (B_{i + 1} \setminus B_{i+1,j}) \cup (B_{i+1,j} \setminus B_{i,j}) \cup (B_{i,j} \setminus B_i) \\
&\subseteq& (B_{i + 1} \setminus B_{i+1,j}) \cup (A_{i+1} \setminus A_i) \cup (B_{i,j} \setminus B_i).
\end{eqnarray*}
Since $\nu$ is a submeasure then
$$
\nu(C_{i+1} \setminus C_i) \leq \nu(B_{i+1} \setminus B_{i+1,j}) + \nu(A_{i+1} \setminus A_i) + \nu(B_{i,j} \setminus B_i) < 2^{-i} 
$$
for sufficiently large $j\in \omega$ 
and thus 
$
\sum_{i} \nu(C_{i+1} \setminus C_i) < \infty. 
$ 
By the standing assumption, there exists $A \in \Sigma$ such that $\nu(C_i \setminus A) = 0$ for all $i \in \omega$, and $\lim_i \nu(A \setminus C_i)= 0$. In particular, 
\begin{equation}\label{eq:firstinequality4}
\lim_{i\to \infty} \nu(C_i \bigtriangleup A)= 0.
\end{equation}

To conclude, we claim that $(A_n)$ is $d_\nu$-convergent to $A$. In fact, observe that for all $i,j \in \omega$ with $i\le j$, we have 
$$
A_i \triangle A\subseteq (A_i \triangle B_{i,j})\cup (B_{i,j} \triangle B_i)\cup (B_i \triangle C_i)\cup (C_i \triangle A).
$$
Since $\nu$ is a submeasure, it follows by inequalities \eqref{eq:firstinequality}, \eqref{eq:firstinequality2}, and \eqref{eq:firstinequality3} that 
$$
\forall i \in \omega, \quad 
\nu(A_i \triangle A)< 2^{1-i}+\nu(C_i \triangle A).
$$
Finally, it follows by \eqref{eq:firstinequality4} that $\lim_i\nu(A_i \triangle A)=0$. Hence $(\Sigma, d_{\nu})$ is complete. 
\end{proof}

We continue with the proofs of its applications.  
\begin{proof}
   [Proof of Corollary \ref{cor:firstthmgangopadyfields}] 
   Pick an increasing sequence $(A_n)$ in $\Sigma$. Observe that $\sum_n \mu(A_{n+1}\setminus A_n)<\infty$: otherwise there would exist $n_0\in \omega$ such that $\mu(X)<\sum_{n\le n_0}\mu(A_{n+1}\setminus A_n)=\mu(A_{n_0+1}\setminus A_0)\le \mu(A_{n_0+1})$, which is impossible. Suppose that there exists $A \in \Sigma$ such that $\mu(A_n\setminus A)=0$ for all $n \in \omega$ and $\lim_n \mu(A\setminus A_n)=0$. Since $|\mu(A)-\mu(A_n)|\le \mu(A_n\setminus A)+\mu(A\setminus A_n)$ for all $n \in \omega$, then $\lim_n\mu(A_n)=\mu(A)$. Conversely, suppose that there exists $A \in \Sigma$ such that $\mu(A_n\setminus A)=0$ for all $n \in \omega$ and $\lim_n\mu(A_n)=\mu(A)$. Since $\mu(A\setminus A_n)=\mu(A)-\mu(A_n)+\mu(A_n\setminus A)$, then $\lim_n \mu(A\setminus A_n)=0$. This proves that item \ref{item:2gang} of Corollary \ref{cor:firstthmgangopadyfields} is a rewriting of item \ref{item:2jon} of Theorem \ref{thm:joncharacterization} for nonnegative finitely additive bounded maps. The conclusion follows by Theorem \ref{thm:joncharacterization}.
\end{proof}

\begin{proof}
[Proof of Corollary \ref{cor:sigmasubadditive}]
    Pick an increasing sequence $(A_n)$ in $\Sigma$ such that $\sum_n \nu(A_{n+1}\setminus A_n)<\infty$ and define $A:=\bigcup_k A_k \in \Sigma$. Then $A_n\setminus A=\emptyset$ for all $n \in \omega$. In addition, since $A\setminus A_n=\bigcup_{k\ge n} (A_{k+1}\setminus A_k)$ for all $n \in \omega$, 
    we obtain 
    $$
    \limsup_{n\to \infty} \nu(A\setminus A_n) \le \limsup_{n \to \infty} \sum_{j\ge n}\nu(A_{j+1}\setminus A_{j})=0.
    $$
    The conclusion follows by Theorem \ref{thm:joncharacterization}. 
\end{proof}

\begin{rmk}\label{rmk:nostrongercondition}
    Item \ref{item:2jon} in Theorem \ref{thm:joncharacterization} cannot be replaced by: 
    \begin{enumerate}[label={\rm (\roman{*}$^\prime$)}]
    \setcounter{enumi}{1}
\item \label{item:2jonprime}For every increasing sequence $(A_n: n \in\omega)$ in $\Sigma$ such that $\sum_n \nu(A_{n+1}\setminus A_n)<\infty$, there exists $A \in \Sigma$ such that $A_n\setminus A \in \mathrm{Fin}$ for all $n \in \omega$ and $\lim_n \nu(A\setminus A_n)=0$\textup{.} 
\end{enumerate}
In fact, choose $\Sigma=\mathcal{P}(\omega)$ and let $\mu$ be as in the statement of Theorem \ref{thm:blass1}. Then $(\mathcal{P}(\omega), d_\mu)$ is complete, i.e., item \ref{item:1jon} in Theorem \ref{thm:joncharacterization} holds. On the other hand, there exists an increasing sequence $(A_n: n \in\omega)$ in $\mathcal{P}(\omega)$ such that, if $A \subseteq \omega$ satisfies $A_n \setminus A\in \mathrm{Fin}$ for all $n \in \omega$, then $\mu(A)\neq \lim_n \mu(A_n)$. Since $\mu$ is invariant modulo finite sets then there exists $\kappa>0$ such that $\mu(A)\ge \mu(A_n)+\kappa$ for all $n \in \omega$. Since $\mu(A\setminus A_n)=\mu(A)-\mu(A_n)+\mu(A_n\setminus A)$ for all $n \in \omega$, we obtain $\liminf_n \mu(A\setminus A_n) \ge \kappa$. 
In addition, $\sum_n \mu(A_{n+1}\setminus A_n)<\infty$, by the same argument as in the proof of Corollary \ref{cor:firstthmgangopadyfields}. Therefore item \ref{item:2jonprime} above fails. 
\end{rmk}

\begin{rmk}
   In the same vein as Remark \ref{rmk:nostrongercondition}, there exists a finitely additive map $\mu: \mathcal{P}(\omega)\to [0,1]$ 
   satisfying property \textbf{\textsc{AP(null)}} and such that 
   its zero set $\mathcal{I}:=\{A\subseteq \omega: \mu(A)=0\}$ is not a $P$-ideal. 
   
   In fact, let $\mu$ be an additive extension of the asymptotic density $\mathsf{d}$ satisfying the claim of Theorem \ref{thm:blass1}. Then by item \ref{item:blass01} the induced pseudometric $d_\mu$ is complete, hence \textbf{\textsc{AP(null)}} holds by Corollary \ref{cor:firstthmgangopadyfields}. In addition, by item \ref{item:blass02}, it is possible to pick an increasing sequence $(A_n: n \in\omega)$ in $\mathcal{P}(\omega)$ such that, if $B\subseteq \omega$ satisfies $A_n\setminus B \in \mathrm{Fin}$ for all $n \in \omega$, then $\mu(B)\neq \lim_n \mu(A_n)$. Thanks to \textbf{\textsc{AP(null)}}, there exists $A\subseteq \omega$ such that $C_n:=A_n\setminus A \in \mathcal{I}$ for all $n \in \omega$ and $\mu(A)=\lim_n \mu(A_n)$. Suppose for the sake of contradiction that $\mathcal{I}$ is a $P$-ideal. Then there exists $C\in \mathcal{I}$ such that $C_n\setminus C \in \mathrm{Fin}$ for all $n \in \omega$. Lastly, define $B:=A\cup C$. Then $A_n\setminus B=C_n\setminus C \in \mathrm{Fin}$ for all $n \in \omega$, and $\mu(B)=\mu(A)=\lim_n\mu(A_n)$. This provides the claimed contradiction. 
\end{rmk}

It is natural to investigate whether the topological and uniform properties of the pseudometric $d_{\|\cdot\|_\varphi}$ induced by an lscsm $\varphi$ can be deduced from some property of the exhaustive ideal $\mathrm{Exh}(\varphi)$. For example, below we identify two lscsms with the same exhaustive ideal that also induce metrically equivalent pseudometrics.

\begin{example}\label{example:Zrepresentation}
    Define $I_n:=\omega \cap [2^n,2^{n+1})$ for all $n \in \omega$, and let $\varphi: \mathcal{P}(\omega)\to [0,\infty]$ and $\psi: \mathcal{P}(\omega)\to [0,\infty]$ be the lscsms given by
    $$
    \varphi(A):=\sup_{n \ge 1}\,\frac{|A\cap n|}{n}
    \quad \text{ and }\quad 
    \psi(A):=\sup_{n \in \omega}\,\frac{|A\cap I_n|}{|I_n|}
    $$
for all $A\subseteq \omega$. 
It is clear that $\|A\|_\varphi=\mathsf{d}^\star(A)$ and 
    $
    \|A\|_\psi=\limsup_n\, |A\cap I_n|/|I_n|
    $ 
    for each $A\subseteq \omega$. Note 
    $$
    \mathcal{Z}=\mathrm{Exh}(\varphi) = \mathrm{Exh}(\psi)
    $$
    (see \cite[Lemma 3.1]{MR3391516} 
    or \cite[Theorem 2]{MR4265483}; 
    cf. also the proof of \cite[Theorem 1.13.3(a)]{MR1711328}, which shows that $\mathcal{Z}$ is a density-ideal). 
    
    In addition, $d_{\|\cdot\|_\psi}$ and $d_{\mathsf{d}^\star}$ are metrically equivalent or, more precisely,  
    $$
    \forall A\subseteq \omega, \quad 
    \frac{\|A\|_\psi}{2} \le \mathsf{d}^\star(A) \le 16\|A\|_\psi.
    $$
    For, fix $A\subseteq \omega$ such that $\alpha:=\mathsf{d}^\star(A) \in (0,1]$. On the one hand, for each $\varepsilon>0$ there exist infinitely many $n \in \omega$ such that $|A\cap I_n|\ge 2^n(1-\varepsilon)\|A\|_\psi$. 
    Since $A \cap I_n\subseteq A\cap 2^{n+1}$, it follows that
    $|A\cap 2^{n+1}|\ge 2^n(1-\varepsilon)\|A\|_\psi$. 
    Since $\varepsilon>0$ is arbitrary, we obtain $\|A\|_\psi\le 2\mathsf{d}^\star(A)$. 
    On the other hand, observe that there exists $n_0 \in \omega$ such that $|A\cap n| \le \frac{5}{4}\alpha n$ for all $n\ge n_0$. Also, there exists an infinite set $S\subseteq \omega$ such that $\min S\ge 2n_0$ and $|A\cap n| \ge \frac{3}{4}\alpha n$ for all $n \in S$. Now, pick $n \in S$ and $m \in \omega$ such that $n \in I_m$. Since $\frac{n}{2} \ge n_0$ it follows that $|A\cap \frac{n}{2}|\le \frac{5}{8}\alpha n$, hence $|A\cap [\frac{n}{2}, n)| \ge \frac{1}{8}\alpha n$. Thus $\max\{|A\cap I_{m-1}|,|A\cap I_m|\}\ge \frac{1}{16}\alpha n \ge \frac{1}{16}\,2^m \alpha$. Hence there exist infinitely many $k\in \omega$ such that $|A \cap I_k| \ge \frac{1}{16}\alpha|I_k|$, so that $\|A\|_\psi\ge \frac{1}{16}\mathsf{d}^\star(A)$. Consequently, $d_{\|\cdot\|_\psi}$ is complete if and only if $d_{\mathsf{d}^\star}$ is complete (cf. Theorem~\ref{thm:jonPAMS}). 
    \end{example}

However, it is not always the case that lscsms with the same exhaustive ideal induce metrically equivalent pseudometrics. 
\begin{prop}\label{prop:notmetricalequivalence}
There exists a lscsm $\varphi: \mathcal{P}(\omega)\to [0,\infty]$ such that\textup{:}
\begin{enumerate}[label={\rm (\roman{*})}]
\item \label{item:1varphioo} $\mathrm{Exh}(\varphi)=\mathcal{Z}$\textup{;}
\item \label{item:2varphioo} $\|\cdot\|_\varphi$ is an upper density\textup{;}
\item \label{item:3varphioo} $d_{\|\cdot\|_\varphi}$ is not metrically equivalent to $d_{\mathsf{d}^\star}$\textup{.}
\end{enumerate}
\end{prop} 
\begin{proof}
    For each real $\alpha>-1$, let $\varphi_\alpha: \mathcal{P}(\omega)\to [0,\infty]$ be the lscsm defined by 
    $$
    \forall A\subseteq \omega, \quad 
    \varphi_\alpha(A):=\sup_{n\ge 1}\frac{\sum_{i \in A \cap [1,n]}i^\alpha }{\sum_{i=1}^n i^\alpha}.
    $$
    Then $\varphi_0$ coincides with the lscsm $\varphi$ used in Example \ref{example:Zrepresentation} and  $\|\cdot\|_{\varphi_\alpha}$ is the classical upper $\alpha$-density, cf. \cite[Example 4]{MR4054777}. In addition, it follows by \cite[Theorem 4.1 and Theorem 4.3]{MR3278191} that $\mathcal{Z}=\mathrm{Exh}({\varphi_\alpha})$ for each $\alpha>-1$. Thus, define the map $\varphi_\infty: \mathcal{P}(\omega)\to [0,\infty]$ by
    $$
    \forall A\subseteq \omega, \quad 
    \varphi_\infty(A):=\sum_{\alpha \in \omega}\frac{\varphi_{2^\alpha}(A)}{2^\alpha}.
    $$
    Then $\varphi_\infty$ is a lscsm that satisfies our claim. 

    It is immediate to see that $\varphi_\infty$ is monotone, subadditive, and $\varphi_\infty(\omega)=1$. 
    Moreover, for each $\varepsilon>0$ and $\alpha >-1$ there exists $n_{\alpha,\varepsilon} \in \omega$ such that $\varphi_{2^\alpha}(A)\le \varphi_{2^{\alpha}}(A \cap n_{\alpha,\varepsilon}) +\varepsilon$ for $n \ge n_{\alpha,\varepsilon}$ (since $\varphi_{2^\alpha}$ is an lscsm). Now, fix $\varepsilon>0$, pick $\alpha_0 \in \omega$ such that $2^{-\alpha_0}\le \varepsilon/3$, and define $n_\varepsilon:=\max\{n_{\alpha,\varepsilon/3}: \alpha \in \alpha_0+1\}$. It follows that 
    \begin{displaymath}
        \begin{split}
            \varphi_\infty(A) &
    \le \sum_{\alpha=0}^{\alpha_0}\frac{\varphi_{2^\alpha}(A)}{2^\alpha}+\sum_{\alpha>\alpha_0}\frac{1}{2^{\alpha}}
    = \sum_{\alpha=0}^{\alpha_0}\frac{\varphi_{2^\alpha}(A)}{2^\alpha}+\frac{1}{2^{\alpha_0}} \\
    &\le \sum_{\alpha=0}^{\alpha_0}\frac{\varphi_{2^\alpha}(A \cap n_{\alpha,\varepsilon/3})+\varepsilon/3}{2^\alpha}+\frac{\varepsilon}{3}\\    
    &\le \sum_{\alpha=0}^{\alpha_0}\frac{\varphi_{2^\alpha}(A \cap n_{\varepsilon})}{2^\alpha}+\frac{\varepsilon}{3}\left(1+\sum_{\alpha=0}^{\alpha_0}\frac{1}{2^\alpha}\right)\le \varphi_\infty(A\cap n_\varepsilon)+\varepsilon. 
        \end{split}
    \end{displaymath}
    This proves that $\varphi_\infty$ is a (bounded) lscsm. 
    
    By the Dominated Convergence theorem, we have that 
    \begin{equation}\label{eq:varphiooipper}
    \forall A\subseteq \omega, \quad 
    \|A\|_{\varphi_\infty}=\lim_{n\to \infty}\sum_{\alpha \in \omega}\frac{\varphi_{2^\alpha}(A\setminus n)}{2^{\alpha}}=\sum_{\alpha \in \omega}\frac{\|A\|_{\varphi_{2^\alpha}}}{2^\alpha}. 
    \end{equation}  
    Since each $\|\cdot\|_{\varphi_\alpha}$ is an upper density and the set of upper densities is countably convex by \cite[Proposition 10]{MR4054777}, then $\|\cdot\|_{\varphi_\infty}$ is an upper density as well, i.e., item \ref{item:2varphioo} holds. 
    Moreover, by the above observations, identity \eqref{eq:varphiooipper} implies that $\mathcal{Z}=\mathrm{Exh}(\varphi_\infty)$, i.e., item \ref{item:1varphioo} holds. 
    
    Lastly, we claim that 
    there is no constant $C>0$ such that $\|A\|_{\varphi_\infty} \le C \|A\|_\psi$ for all $A\subseteq \omega$, where $\psi$ is the lscsm used in Example \ref{example:Zrepresentation}. For, define $A_n:=\omega \cap \bigcup_{i \in \omega}[2^{i},2^i(1+2^{-n}))$ for each $n \in \omega$, and observe that $\|A_n\|_\psi=2^{-n}$. Then, for each $n \in \omega$ and each $\alpha>-1$, we obtain
\begin{displaymath}
    \begin{split}
        \|A_n\|_{\varphi_{\alpha}}&= \limsup_{k\to \infty}\frac{\sum_{i \in A_n \cap [1,2^k(1+2^{-n})]}i^{\alpha}}{\sum_{i=1}^{2^k(1+2^{-n})} i^{\alpha}}\\
        &\ge \limsup_{k\to \infty}\frac{\sum_{i =2^k}^{2^k(1+2^{-n})}i^{\alpha}}{\sum_{i=1}^{2^k(1+2^{-n})} i^{\alpha}}
        =1-(1+2^{-n})^{-{\alpha}-1}. 
    \end{split}
\end{displaymath}
Putting it all together, it follows that 
    \begin{equation}\label{eq:liminffinalclaim}
        \forall \alpha >-1,
        \quad 
        \liminf_{n \to \infty}
        \frac{\|A_n\|_{\varphi_\alpha}}{\|A_n\|_{\psi}}
        \ge 
        \lim_{n \to \infty} 
        \frac{1-(1+2^{-n})^{-\alpha-1}}{2^{-n}}=\alpha+1.
\end{equation}
Now, fix $C \in \omega$. Thanks to \eqref{eq:liminffinalclaim}, there exists a sufficiently large $n_C \in \omega$ such that 
$$
\forall \alpha \in 2C, \quad \frac{\|A_{n_C}\|_{\varphi_{2^\alpha}}}{\|A_{n_C}\|_\psi} \ge \frac{2^\alpha+1}{2}.
$$
Since $\frac{2^\alpha+1}{2}> 2^{\alpha-1}$, it follows by \eqref{eq:varphiooipper} that 
$$
\|A_{n_C}\|_{\varphi_\infty} \ge \sum_{\alpha \in 2C}\frac{\|A_{n_C}\|_{\varphi_{2^\alpha}}}{2^\alpha}
>\|A_{n_C}\|_\psi \sum_{\alpha\in 2C}\frac{2^{\alpha-1}}{2^\alpha}=C \|A_{n_C}\|_\psi.
$$
Therefore $d_{\|\cdot\|_{\varphi_\infty}}$ and $d_{\|\cdot\|_{\psi}}$ are not metrically equivalent. Since $d_{\|\cdot\|_{\psi}}$ and $d_{\mathsf{d}^\star}$ are metrically equivalent by Example \ref{example:Zrepresentation}, we conclude that item \ref{item:3varphioo} holds. 
\end{proof}

It is true that lscsms with the same exhaustive ideal induce pseudometrics that are topologically equivalent.
\begin{prop}\label{rmk:topequivalent}
Let $\varphi_1,\varphi_2$ be two lscsms such that $\mathrm{Exh}(\varphi_1)=\mathrm{Exh}(\varphi_2)$. Then the pseudometrics $d_{\|\cdot\|_{\varphi_1}}$ and $d_{\|\cdot\|_{\varphi_2}}$ are topologically equivalent.
\end{prop}
\begin{proof}
It is easy to check that $\mathrm{Exh}(\varphi_1)=\mathrm{Exh}(\varphi_2)$ 
is equivalent to the fact that,
for every disjoint sequence $(F_n)$ of finite sets of $\omega$, we have $\lim_n\varphi_1(F_n)=0$ 
if and only if 
$\lim_n \varphi_2(F_n)=0$, cf. \cite[Remark 2.3]{MR3436368}. 
Now, pick a sequence of sets $(A_n)$ and $A\subseteq \omega$. Since 
$\varphi_1$ and $\varphi_2$ are two lscsms, there exists a partition $(P_n: n \in \omega)$ of $\omega$ into nonempty finite sets such that 
$$
|\varphi_i((A_n\bigtriangleup A) \cap P_n)-\|A_n\bigtriangleup A\|_{\varphi_i}|\le 2^{-n}
$$
for each $i \in \{1,2\}$ and for all 
sufficiently large $n$. 
It follows that $\lim_{n} \|A\bigtriangleup A_n\|_{\varphi_1}=0$ if and only if $\lim_{n} \|A\bigtriangleup A_n\|_{\varphi_2}=0$, concluding the proof. 
\end{proof}

Next we prove Proposition \ref{prop:positivesimpleexample} and Theorem~\ref{thm:alllscsm} using Theorem~\ref{thm:joncharacterization}. 

\begin{proof}
    [Proof of Proposition \ref{prop:positivesimpleexample}] 
    Let $(A_n: n \in \omega)$ be an increasing sequence in $\mathcal{P}(\omega)$ such that $\sum_n \bm{1}_{\mathcal{I}^+}(A_{n+1}\setminus A_n)<\infty$. Define $S:=\{n \in \omega: A_{n+1}\setminus A_n \in \mathcal{I}^+\}$, so that $S$ is necessarily finite by the standing hypothesis. Define $A:=A_0$ if $S=\emptyset$, otherwise $A:=A_{1+\max S}$. It easily follows that $\nu(A_n\setminus A)=0$ for all $n \in \omega$, and that the real sequence $(\nu(A\setminus A_n): n \in \omega)$ is eventually $0$. Therefore item \ref{item:2jon} holds. The claim follows by Theorem \ref{thm:joncharacterization}.
\end{proof}

\medskip

\begin{proof}
[Proof of Theorem \ref{thm:alllscsm}] 
    %
    %
    Let $(A_n: n \in \omega)$ be an increasing sequence of sets such that $\sum_n \|A_{n+1}\setminus A_n\|_\varphi<\infty$. Thanks to Theorem \ref{thm:joncharacterization}, it will be enough to show that there exists $A \subseteq \omega$ such that $A_n\setminus A$ is finite for all $n \in \omega$ and $\lim_n \|A\setminus A_n\|_\varphi=0$ (note that this is stronger than item \ref{item:2jon} in Theorem \ref{thm:joncharacterization}). 
    
    To this aim, define $B_j:=A_{j+1}\setminus A_j$ for all $j \in \omega$. Let $(n_j: j \in \omega)$ be a strictly increasing sequence in $\omega$ such that 
    \begin{equation}\label{eq:firsthypothesislscsm}
    \forall j \in\omega, \quad
     \varphi(B_j\setminus n_j)\le 2\|B_j\|_\varphi. 
    \end{equation}
    We claim that $A:=A_0\cup \bigcup_{j \in \omega}(B_j\setminus n_j)$ satisfies our claim. For, note that $A_k\setminus A \subseteq n_k$ for all $k \in \omega$. To complete the proof, we need to show that $\lim_n \|A\setminus A_n\|_\varphi=0$. To this aim, fix $\varepsilon>0$ and pick $k_0 \in \omega$ such that $\sum_{j\ge k_0}\|B_j\|_\varphi<\varepsilon/4$. Pick an integer $k\ge k_0$. Taking into account that $\|\cdot\|_\varphi$ is invariant under finite modifications and that the elements of $(B_j)$ are pairwise disjoint, it follows that 
\begin{equation}\label{eq:dvarphiequ1}
    \begin{split}
         \|A\bigtriangleup A_k\|_\varphi 
         \le \left\|\,n_k \cup \bigcup_{j\ge k}((B_j\setminus n_j)\setminus A_k)\right\|_\varphi
         =\left\|\,\bigcup_{j\ge k}(B_j\setminus n_j)\right\|_\varphi.
    \end{split}
\end{equation}
Since $\varphi$ is a lscsm, there exists an integer $m_k>n_{k+1}$ such that
\begin{equation}\label{eq:dvarphiequ2}
\left\|\,\bigcup_{j\ge k}(B_j\setminus n_j)\right\|_\varphi \le 2\varphi\left(m_k\cap \bigcup_{j\ge k}(B_j\setminus n_j)\right).
\end{equation}
Fix an integer $i_k \in \omega$ such that $n_{i_k}>m_k$. Since $\min (B_j\setminus n_j) \ge n_j > m_k$ for all $j\ge i_k$, it follows by \eqref{eq:firsthypothesislscsm}, \eqref{eq:dvarphiequ1}, \eqref{eq:dvarphiequ2}, and the above observations that 
\begin{displaymath}
    \begin{split}
        \|A\bigtriangleup A_k\|_\varphi
        &\le 2\varphi\left(\bigcup_{k\le j<i_k}(B_j\setminus n_j)\right)\\
        &\le 2\sum_{k\le j<i_k}\varphi(B_j\setminus n_j)
        \le 2\sum_{j\ge k}\varphi(B_j\setminus n_j)\\
        &\le 4\sum_{j\ge k}\|B_j\|_\varphi
        \le 4\sum_{j\ge k_0}\|B_j\|_\varphi
        < \varepsilon. 
    \end{split}
\end{displaymath}
Therefore $\lim_n \|A\bigtriangleup A_n\|_\varphi=0$.
\end{proof}

\begin{rmk}
Solecki proved also in \cite[Theorem 3.1]{MR1708146} that, if $\varphi: \mathcal{P}(\omega) \to [0,\infty]$ is a lscsm, then $\mathrm{Exh}(\varphi)$ is $d_\varphi$-closed and separable. Since $d_{\|\cdot\|_\varphi}(A,B)\le d_\varphi(A,B)$ for all $A,B\subseteq \omega$, then $\mathrm{Exh}(\varphi)$ is $d_{\|\cdot\|_\varphi}$-closed and separable as well. Alternatively, this can be seen directly: since $\|\cdot\|_\varphi$ is $d_{\|\cdot\|_\varphi}$-continuous, then the $\|\cdot\|_\varphi$-preimage of $\{0\}$, that is, $\mathrm{Exh}(\varphi)$, is closed (in this case, the separability is obvious). 
\end{rmk}

We proceed with the proofs of our main  results.
\begin{proof}
    [Proof of Proposition \ref{prop:negativesimpleexample}] 
    Let us suppose for the sake of contradiction that the space $(\mathcal{P}(\omega), d_{\nu})$ is complete. 
    Set $A_n:=\{0,1,\ldots,n\}$ for each $n \in \omega$, and observe that $\sum_n \nu(A_{n+1}\setminus A_n)\le \sum_n a_n <\infty$. It follows by Theorem \ref{thm:joncharacterization} that there exists $A\subseteq \omega$ such that $\nu(A_n\setminus A)=0$ for all $n \in \omega$ and $\lim_n \nu(A\setminus A_n)=0$. By the definition of $\nu$, the former condition implies that $A_n\subseteq A$ for all $n \in \omega$. Hence necessarily $A=\omega$. On the other hand, we have 
    $\nu(A\setminus A_n)
    \ge \bm{1}_{\mathcal{I}^+}(\omega\setminus n+1)=1$ for all $n \in \omega$. This provides the claimed contradiction.
\end{proof}

\medskip
\begin{proof}
[Proof of Theorem \ref{thm:casemustarbiggerupperBanach}] 
Let $\mu^\star$ be an upper density on $\omega$ such that $\mathsf{bd}^\star(A)\le \mu^\star(A)$ for all $A\subseteq \omega$. Thanks to Theorem \ref{thm:joncharacterization}, it will be enough to show that for each $\kappa \in (0,1)$ there exists an increasing sequence $(A_n: n\in \omega)$ of subsets of $\omega$ such that $\sum_n \mu^\star(A_{n+1}\setminus A_n)<\infty$ and, if a set $A\subseteq \omega$ satisfies $\mu^\star(A_n\setminus A)=0$ for all $n \in \omega$, then $\mu^\star(A\setminus A_n)\ge \kappa$ for all $n \in \omega$.   

    To this aim, fix $\kappa \in (0,1)$ and a positive integer $N>e^2$ such that 
    \begin{equation}\label{eq:firstcondition}
    \frac{\log^2(N)}{N}<\frac{1-\kappa}{2}.
    \end{equation}
    Define the positive integer $C:=\lceil 1/\kappa\rceil$ and let $(a_n: n \in \omega)$ be a strictly increasing sequence of positive integers with the property that 
    \begin{equation}\label{eq:definitionkappa}
    \sum_{j \in \omega}\frac{j}{a_j!}<\frac{1-\kappa}{2} 
    \quad \text{ and }\quad 
     a_n!>\max\left\{\frac{n}{\kappa},\, Ce^{2(n+1)},\, (N+1)^2 \right\} 
    \end{equation}
    for all $n \in \omega$. 
    Thus, define recursively three sequences of sets of nonnegative integers $(A_n: n\in \omega)$, $(B_n: n \in \omega)$ and $(H_n: n \in \omega)$ as follows: 
    \begin{enumerate}[label={\rm (\roman{*})}]
    \item \label{item:point1} Set $H_0:=B_0:=\emptyset$.

    \smallskip
    
    \item \label{item:point2}Set $H_1:=\{0\}$, $A_0:=B_1:=a_1!\cdot (\omega\setminus \{0\})$.

    \smallskip
    
    Observe that $A_0=a_2!\cdot \omega+\{a_1!j: j \in \{1,2,\ldots,a_2!/a_1!\}\}$ and $\mathsf{d}(A_0)=\mathsf{d}(B_1)=1/a_1!$; in addition, we have $|A_0\cap n|/n<1/a_1!$ for all $n\ge 1$, hence choosing $n=a_2!$ we get 
    $$
    |a_2!\setminus A_0|>a_2!(1-1/a_1!) \ge a_2> 2.
    $$
   \item \label{item:point3} Given a positive integer $n\ge 1$, suppose that all sets $(A_i: i \in n)$, $(B_i: i \in n+1)$, and $(H_i: i \in n+1)$ have been defined satisfying the following properties: 
   \begin{enumerate}
       \item \label{item:alonglist} all sets $B_i$ are pairwise disjoint; 
       \item \label{item:blonglist} $H_i\subseteq a_i!$ and $|H_i|=i$ for all $i \in n+1$; 
       \item \label{item:clonglist} $i \ge \kappa \ell_i$ for all nonzero $i \in n+1$, where 
       $$
       \ell_i:=1+\max H_i-\min H_i; 
       $$ 
       \item \label{item:dlonglist} $B_i=a_i!\cdot (\omega\setminus \{0\})+H_i$ for all $i \in n+1$; 
       \item \label{item:elonglist} $A_i=B_0\cup \cdots \cup B_{i+1}$ for all $i\in n$;
       \item \label{item:flonglist} $|A_i \cap m|/m < \mathsf{d}(A_i)=\sum_{j\in i+2} j/a_j!$ for all $m\ge 1$ and $i \in n$; 
   \end{enumerate}
   (In particular, such conditions hold for $n=1$, thanks to item \ref{item:point2}.)
   
\smallskip

   \item \label{item:point4} It follows by item \ref{item:point3} that there exists a unique $K_{n+1}\subseteq a_{n+1}!$ such that 
   $$
   A_{n-1}\setminus a_{n+1}!=\bigcup_{i \in n+1}B_i\setminus a_{n+1}!=a_{n+1}!\cdot (\omega\setminus \{0\})+K_{n+1}
   $$
   and, by hypothesis \eqref{eq:definitionkappa}, also $$
   a_{n+1}!-|K_{n+1}|=a_{n+1}!\left(1-\sum_{i \in n+1}\frac{i}{a_i!}\right)\ge \kappa a_{n+1}!>n+1.
   $$
   Hence it is possible to pick a set $H_{n+1}\subseteq a_{n+1}! \setminus K_{n+1}$ such that $|H_{n+1}|=n+1$ and $\max H_{n+1}$ is minimized. 

\smallskip

      \item \label{item:point5} Lastly, define 
      \begin{equation}\label{eq:definitionBniaAn}
      B_{n+1}:=a_{n+1}!\cdot (\omega\setminus \{0\})+H_{n+1}
      \,\, \text{ and }\,\,
      A_n:=\bigcup_{i \in n+2}B_i.
      \end{equation}
  To complete the induction, we have to show that the conditions in item \ref{item:point3} hold for $n+1$. For, observe by item \ref{item:point4} that $B_{n+1}\cap A_{n-1}=\emptyset$, hence $B_{n+1}\cap B_i=\emptyset$ for all $i\in n+1$. Together with the inductive step, this proves condition (a). Condition (b) follows by the construction in item \ref{item:point4}, while conditions (d) and (e) follow by the above definitions of $B_{n+1}$ and $A_n$ in \eqref{eq:definitionBniaAn}. Since $B_{n+1}$ is a finite union of infinite arithmetic progressions (i.e., $B_{n+1} \in \mathscr{A}$) and it is disjoint from $A_{n-1}$, we obtain 
  $$
  \mathsf{d}(A_n)=\mathsf{d}(A_{n-1})+\mathsf{d}(B_{n+1})=\sum_{j \in n+1}\frac{j}{a_j!}+\frac{|H_{n+1}|}{a_{n+1}!}=\sum_{j \in n+2}\frac{j}{a_j!},
  $$
  cf. also \cite[Proposition 7 and Proposition 8]{MR4054777}. In addition, by the definitions in \eqref{eq:definitionBniaAn}, it is not difficult to observe (since $B_i \cap a_i!=\emptyset$ for all $i$) that $|A_{n} \cap m|/m< \mathsf{d}(A_n)$ for all $m\ge 1$, proving also condition (f). 
  Now, note that for all nonzero $i \in n+2$ we have by \eqref{eq:definitionkappa} that 
  \begin{displaymath}
  \begin{split}
      \ell_i &\le \max H_i \le \sum_{j=0}^i j \left\lceil \frac{\ell_i}{a_j!}\right\rceil\\
  &\le \sum_{j=0}^i j +\ell_i \sum_{j \in \omega}\frac{j}{a_j!} 
  \le i^2+\ell_i (1-\kappa),
  \end{split}
  \end{displaymath}
  which implies that 
  $$
  \ell_i \le \frac{i^2}{\kappa} 
  \quad \text{ and }\quad 
  \max H_i \le i^2+\ell_i (1-\kappa) \le \frac{i^2}{\kappa}.
  $$
  Recalling that $C=\lceil 1/\kappa\rceil$, it follows that 
  \begin{equation}\label{eq:splitAnminus1}
  \forall i \in n+2, \quad 
  B_i\subseteq a_i!\cdot \omega + \{0,1,\ldots,Ci^2\}. 
  \end{equation}
  In addition, since $a_0!>(N+1)^2$ by \eqref{eq:definitionkappa}, it follows by the construction of the sets $H_i$ that 
  $$
  \forall i \in N+2, \quad 
  B_i=a_i!\cdot (\omega\setminus 0)+\left\{\binom{i}{2}+j: j \in i\right\}.
  $$
  The latter identity implies that, if $n \in \{1,\ldots,N\}$, then $A_{n-1}\subseteq a_0!\cdot \omega + [0,1,\ldots, \max H_n]$ and $(a_{n+1}!+[\min H_{n+1}, \max H_{n+1}])\subseteq a_0!\cdot \omega + [\min H_{n+1},\max H_{n+1}]$. Considering also that $\max H_n<\min H_{n+1}$ by construction, we get
  $$
  \forall n \in \{1,\ldots,N\}, \quad 
  A_{n-1} \cap (a_{n+1}!+[\min H_{n+1}, \max H_{n+1}])=\emptyset. 
  $$
  In particular, if $n \in \{1,\ldots,N\}$, then $\ell_{n+1}=n+1$. 
  Suppose now that $n\ge N+1$ and note by \eqref{eq:firstcondition} that $\log^2(n)/n \le \log^2(N)/N< (1-\kappa)/2$. Observe that, if $i> \log(n)$ then $i> \log((n+1)/e)$, which is equivalent to $n+1<e^{i+1}$. Hence it follows by \eqref{eq:definitionkappa} and \eqref{eq:splitAnminus1} that, if $i> \log(n)$, then 
  $$
  \max H_{n+1}\le C(n+1)^2 < Ce^{2(i+1)} \le a_i!.
  $$
  This implies that 
  \begin{displaymath}
      \begin{split}
          |A_{n-1}\cap (a_{n+1}!&+[\min H_{n+1}, \max H_{n+1}])|
          \\
          &\le \sum_{i=0}^n |B_i\cap (a_{n+1}!+[\min H_{n+1}, \max H_{n+1}])|\\
          &\le \ell_{n+1}\left(\sum_{i=0}^n \frac{i}{a_i!}\right)+\sum_{i=0}^{\lfloor \log(n)\rfloor }i\\
          &\le \ell_{n+1} \left(\sum_{i\in \omega} \frac{i}{a_i!}+ \frac{\log^2(n)}{\ell_{n+1}}\right)\\
          &\le \ell_{n+1}\left(\frac{1-\kappa}{2}+\frac{\log^2(n)}{n}\right)\\
          &\le \ell_{n+1}(1-\kappa).
      \end{split}
  \end{displaymath}
Hence, in both cases (i.e., if $n\le N$ or if $n\ge N+1$), we get that $|A_{n-1}\cap (a_{n+1}!+[\min H_{n+1}, \max H_{n+1}])|\le \ell_{n+1}(1-\kappa)$. Therefore
 \begin{displaymath}
      \begin{split}
     n+1&=|H_{n+1}|=|B_{n+1}\cap (a_{n+1}!+[\min H_{n+1}, \max H_{n+1}])|\\
     &=\ell_{n+1}-|A_{n-1}\cap (a_{n+1}!+[\min H_{n+1}, \max H_{n+1}])|\ge \kappa \ell_{n+1}. 
      \end{split}
\end{displaymath}
%
%
  %
  %
  %
  %
  %
 %
This proves condition (c) 
and completes the induction. 
   \end{enumerate}

\medskip

To complete the proof, we claim that the sequence of sets $(A_n: n \in \omega)$ in the above construction satisfies our claim. For, we have by construction that $A_{n+1}\setminus A_n=B_{n+2} \in \mathscr{A}$ for all $n \in \omega$. 
Since $\mu^\star(B_{n})=|H_n|/a_n!$ by \cite[Proposition 7]{MR4054777}, it follows by \eqref{eq:definitionkappa} and the conditions in item \ref{item:point5} that 
$$
\sum_{n \in\omega}\mu^\star(A_{n+1}\setminus A_n) \le \sum_{n \in \omega}\mu^\star(B_n)=\sum_{n \in\omega}\frac{n}{a_n!}<\infty. 
$$

Now, pick a set $A\subseteq \omega$ such that $\mu^\star(A_n\setminus A)=0$ for all $n \in \omega$. Define also 
$$
\forall n \in \omega, \quad C_n:=B_n\setminus A \quad \text{ and }\quad D_n:=B_n\setminus C_n, 
$$
so that $\{C_n,D_n\}$ is a partition of $B_n$, and $D_n\subseteq A$ for all $n\in \omega$. 
In addition, we get by monotonicity of $\mu^\star$ that $\mu^\star(C_n)\le \mu^\star(A_{n}\setminus A)=0$ for all $n \in \omega$, hence $\mu^\star(C_n)=0$. 
Thanks to \eqref{eq:inequalitiesupperdensities}, we get $\mathsf{d}^\star(C_n)=0$, i.e., 
\begin{equation}\label{eq:limitdensity}
\forall n \in \omega, \forall \varepsilon>0, \exists N_{n,\varepsilon}\ge 1, \forall m \ge N_{i,\varepsilon},
\quad 
\frac{|C_n \cap m|}{m}<\varepsilon.
\end{equation}

It follows that for any positive integer $n$ and sufficiently large $m$ we have $|C_n \cap m|/m < 1/a_n!$. 
This implies there exists 
$j_n \in \omega$
such that 
$$
C_n \cap (a_n!j_n+[\min H_n, \max H_n])=\emptyset.
$$
Hence, by the condition (c) in item \ref{item:point3}, we get
$$
\frac{|D_n \cap (a_n!j_n+[\min H_n, \max H_n])|}{\ell_n}
=\frac{|H_n|}{\ell_n}\ge \kappa. 
$$
It follows by the definition of the upper Banach density that 
\begin{equation}\label{eq:bdstarDi}
\forall n\ge 1, \quad \mathrm{bd}^\star\left(\bigcup_{i\ge n}D_i\right)\ge \kappa. 
\end{equation}

Therefore, for each $n \in\omega$, we get 
$$
\mu^\star(A\setminus A_n) \ge \mu^\star\left(\bigcup_{i\ge n+2}D_i\right)\ge \mathrm{bd}^\star\left(\bigcup_{i\ge n+2}D_i\right)\ge \kappa, 
$$
which completes the proof. 
\end{proof}

\begin{rmk}\label{rmk:upperbanach}
The argument provided in the proof of Theorem \ref{thm:casemustarbiggerupperBanach} shows that, if $\mathscr{A}_\infty$ stands for the family of countably infinite unions of infinite arithmetic progressions $k\cdot \omega+h$, then 
$$
\mathscr{A}_\infty\setminus \mathsf{dom}(\mathsf{bd})\neq \emptyset.
$$ 

In fact, suppose that $\kappa=1/2$, so that $\sum_n n/a_n!\le 1/4$ by \eqref{eq:definitionkappa}. Then the set $B:=\bigcup_n B_n \in \mathscr{A}_\infty$ contains $\bigcup_{i \ge 1}D_n$, hence $\mathsf{bd}^\star(B)\ge 1/2$ by \eqref{eq:bdstarDi}. On the other hand, for each $m\ge 1$, there exists $k_m \ge 1$ for which $B\cap m=\bigcup_{i \in k_m}B_i \cap m$, so that by the definition of the sets $B_i$ we get 
$$
\frac{|B \cap m|}{m} \le \sum_{i \in k_m}\frac{|B_i \cap m|}{m} \le \sum_{i \in k_m}\frac{i}{a_i!}\le \frac{1}{4}.
$$
This implies that $\mathsf{bd}_\star(B) \le \mathsf{d}_\star(B) \le 1/4$, cf. \eqref{eq:inequalitiesupperdensities}. Therefore $B \notin \mathrm{dom}(\mathsf{bd})$. 

This example also shows that the restriction of the upper Banach density $\mathsf{bd}^\star$ to $\mathcal{P}(\omega)\setminus \{A\subseteq \omega: \mathsf{bd}^\star(A)=0\}$ is not $\sigma$-subadditive. This comment applies to all upper densities $\mu^\star$ for which $\mu^\star(A) \ge \mathsf{bd}^\star(A)$ for all $A\subseteq \omega$. 
%
\end{rmk}

\begin{proof}
    [Proof of Lemma \ref{lem:easyimplication}] 
    Suppose that $(\Sigma,d_\nu)$ is complete, and let $(A_n: n \in \omega)$ be a Cauchy sequence in $(\overline{\Sigma}, d_{\nu^\star})$. Passing if needed to a suitable subsequence, we can suppose without loss of generality that $d_{\nu^\star}(A_n,A_m) \le 2^{-\min\{n,m\}}$ for all $n,m \in \omega$. Now, for each $n \in \omega$, the set $A_n$ belongs to $\overline{\Sigma}$, hence there exists $C_n \in \Sigma$ such that $d_{\nu^\star}(A_n,C_n) \le 2^{-n}$. Observe that the sequence $(C_n: n \in\omega)$ is $d_\nu$-Cauchy since, for all $n,m \in\omega$, we have 
    $$
    d_{\nu}(C_n,C_m) \le d_{\nu^\star}(C_n,A_n)+d_{\nu^\star}(A_n,A_m)+d_{\nu^\star}(A_m,C_m) < 2^{2-\min\{n,m\}}. 
    $$
    Since $(\Sigma, d_\nu)$ is complete, the sequence $(C_n: n \in \omega)$ is $d_\nu$-convergent to some $C \in \Sigma$. To complete the proof, it is enough to note that $d_{\nu^\star}(A_n,C) \le d_{\nu^\star}(A_n,C_n)+d_{\nu^\star}(C_n,C)$ for all $n \in \omega$, hence $(A_n: n \in \omega)$ is $d_{\nu^\star}$-convergent to $C$ as well. Therefore $(\overline{\Sigma}, d_{\nu^\star})$ is complete. 
\end{proof}

\medskip

\begin{rmk}\label{rmk:peanocompleation}
    Let $\Sigma$ and $\nu$ be as in the statement of Theorem \ref{thm:joncharacterization}, and let $\overline{\Sigma}$ be the closure of $\Sigma$ in $(\mathcal{P}(X), d_{\nu^\star})$, as in Section \ref{sec:stone}. Then 
    $$
    \overline{\Sigma}=\left\{B\subseteq X: \forall \varepsilon>0, \exists A,C \in \Sigma, \quad A\subseteq B\subseteq C \,\text{ and }\,\nu(C\setminus A)<\varepsilon\right\}.
    $$
    In fact, it is clear that the family on the right hand side is contained in $\overline{\Sigma}$. To show the converse, pick $B \in \overline{\Sigma}$ and $\epsilon > 0$. Then there is $D \in \Sigma$ such that $\nu^\star(B \triangle D) < \epsilon$. By definition of $\nu^\star$, there is $E \in \Sigma$ such that $B \triangle D \subseteq E$ and $\nu(E) < \epsilon$. Define $A := D \setminus E\in \Sigma$ and $C := D \cup E \in \Sigma$. Then $A \subseteq B \subseteq C$ and $\nu(C \setminus A) = \nu(E) < \epsilon$. 
\end{rmk}

\medskip

\begin{lem}\label{lem:closureclopenlemma}
    With the notations of Section \ref{sec:stone}, we have
    $$
    \mathcal{C}\cup \mathcal{N}_{\hat{\nu}} \subseteq 
    \mathrm{Cl}_{\hat{\nu}}(\mathcal{C})
    \subseteq \sigma(\mathcal{C}\cup \mathcal{N}_{\hat{\nu}}).
    $$
\end{lem}
\begin{proof} 
%
The first inclusion is clear. 
To show the second inclusion, pick $A \in \mathrm{Cl}_{\hat{\nu}}(\mathcal{C})$.  
Given $\epsilon > 0$, there is $B \in \mathcal{C}$ and a sequence $(C_n: n \in \omega)$ in $\mathcal{C}$ such that $A \triangle B \subseteq \bigcup_{n} C_n$ and $\sum_{n } \hat{\nu}(C_n) < \epsilon$. Define $C := \bigcup_{n} C_n \in \sigma(\mathcal{C})$, hence $B \setminus C \subseteq A \subseteq B \cup C$ and $\hat{\nu}((B \cup C) \setminus (B \setminus C)) = \hat{\nu}(C) < \epsilon$. Since this holds for all $\epsilon > 0$, there exist $D, E \in \sigma(\mathcal{C})$ such that $D \subseteq A \subseteq E$ and $\hat{\nu}(E \setminus D) = 0$. Therefore $A \in \sigma(\mathcal{C} \cup \mathcal{N}_{\hat{\nu}})$. 
\end{proof}

We remark that $\mathrm{Cl}_{\hat{\nu}}(\mathcal{C})$ is, in general, \emph{not} a $\sigma$-field. More precisely, the following example shows that $\sigma(\mathcal{C})$ might not be contained in $\mathrm{Cl}_{\hat{\nu}}(\mathcal{C})$, even if the underlying pseudometric space is complete. 
\begin{example}\label{example:ClCnotsigmafield}
        Set 
        $X:=\omega$, $\Sigma:=\mathcal{P}(\omega)$, and define  $\nu: \Sigma\to \mathbb{R}$ by 
$$
\forall A\subseteq X, \quad 
\nu(A):=
\begin{cases}
\,0 \,\,\,&\text{if $A$ is finite},\\
\,1 &\text{if $A$ is infinite}.
\end{cases}
$$
Note 
that $(\Sigma, d_\nu)$ is complete. We use the same notations of Section \ref{sec:stone}, e.g., $S$ is the Stone space of 
$\Sigma/\nu$ and $\mathcal{C}$ is its field
of clopen subsets.


If $B=\emptyset$ then $\hat{\nu}(B)=0$. If $B\neq\emptyset$ and $(A_n)\in\Delta(B)$, then at least one $\phi(\pi(A_n))$ is nonempty, hence $\pi(A_n)\neq \pi(\emptyset)$ and $\nu(A_n)=1$. Thus $\sum_n \nu(A_n)\ge 1$, and consequently $\hat{\nu}(B)\ge 1$.
On the other hand, $B\subseteq S=\phi(\pi(X))$ and $\nu(X)=1$, so $\hat{\nu}(B)\le 1$.
It follows that 
$$
\forall B\subseteq S, \quad 
\widehat{\nu}(B)=
\begin{cases}
\,0\,\,\,&\text{ if }B=\emptyset,\\
\,1&\text{ if }B\neq\emptyset.
\end{cases}
$$
This implies that the topology induced by $d_{\hat{\nu}}$ on $\mathcal{P}(S)$ is the discrete one. In particular, $\mathrm{Cl}_{\hat{\nu}}(\mathcal{C})=\mathcal{C}$. Hence, it is enough to show that $\mathcal{C}$ is not a $\sigma$-field. 
To this end, fix a sequence $(I_n: n\in\omega)$ of $\omega$ of pairwise disjoint infinite sets and define $G_n:=\phi(\pi(I_n))$ for each $n \in \omega$. Since $(G_n: n \in \omega)$ is a sequence of pairwise disjoint nonempty clopen subsets of $S$ and $S$ is compact, it follows that $G:=\bigcup_n G_n\notin \mathcal{C}$. Therefore $\mathcal{C}$ is not a $\sigma$-field. 

Taking into account that both $\sigma(\mathcal{C})$ and $\mathrm{Cl}_{\hat{\nu}}(\mathcal{C})$ are stable under complements and choosing $S\setminus G$, 
it follows 
that there exists a complete pseudometric space which admits a compact $G_\delta$ set in $\sigma(\mathcal{C})\setminus \mathrm{Cl}_{\hat{\nu}}(\mathcal{C})$. 
\end{example}

\color{black}
\medskip

Additional properties of $\mathrm{Cl}_{\hat{\nu}}(\mathcal{C})$ are given below, provided that $\nu$ is finite.
\begin{lem}\label{lem:secondrepresentation}
 With the notations of Section \ref{sec:stone}, 
    suppose that $\nu(X)<\infty$, and let $\mathscr{M}_{\hat{\nu}}$ be the family of Carath\'{e}odory $\hat{\nu}$-measurable sets, that is, 
    $$
    \mathscr{M}_{\hat{\nu}}:=\left\{B\subseteq S: \forall C\subseteq S, \quad \hat{\nu}(C)=\hat{\nu}(C\cap B)+\hat{\nu}(C\setminus B)\right\}.
    $$
    Then $\mathscr{M}_{\hat{\nu}}$ is a $\sigma$-field, $\hat{\nu}$ is $\sigma$-additive on $\mathscr{M}_{\hat{\nu}}$, and  
    $\mathcal{N}_{\hat{\nu}}\subseteq \mathscr{M}_{\hat{\nu}}\subseteq \mathrm{Cl}_{\hat{\nu}}(\mathcal{C})$. 
\end{lem}
\begin{proof}
Recall that $\hat{\nu}$ is an outer measure on $\mathcal{P}(S)$ (that is, it is monotone, $\sigma$-additive, and satisfies $\hat{\nu}(\emptyset)=0$). 
    Hence 
    it follows 
    by Carath\'{e}odory's Theorem that $\mathscr{M}_{\hat{\nu}}$ is a $\sigma$-field which contains $\mathcal{N}_{\hat{\nu}}$, and the restriction of $\hat{\nu}$ on $\mathscr{M}_{\hat{\nu}}$ is $\sigma$-additive, 
    cf. e.g. \cite[Theorem 1.11]{MR1681462}. 

        \begin{claim}\label{claim:identitynuhat}
        $\nu=\hat{\nu} \circ \phi \circ \pi$. 
    \end{claim}
        \begin{proof}
Pick $A\in \Sigma$. Since $(A,\emptyset, \emptyset, \ldots) \in \Delta(\phi(\pi(A)))$, we have $\hat{\nu}(\phi(\pi(A))) \le \nu(A)$. Conversely, suppose $(A_n: n \in \omega) \in \Delta(\phi(\pi(A)))$. Since $\{\phi(\pi(A_n)): n \in \omega\}$ is an open cover of the compact set $\phi(\pi(A))$, there exists $F \in \mathrm{Fin}$ such that $\phi(\pi(A))\subseteq \bigcup_{n \in F}\phi(\pi(A_n))$. Since $\phi$ is a Boolean isomorphism and $\nu$ is a submeasure, 
we get
$
\nu(A)\le \sum_{n \in F}\nu(A_n)\le \sum_{n \in \omega}\nu(A_n). 
$
Therefore $\nu(A) \le \hat{\nu}(\phi(\pi(A)))$. 
    \end{proof}
    
    To show the last inclusion $\mathscr{M}_{\hat{\nu}}\subseteq \mathrm{Cl}_{\hat{\nu}}(\mathcal{C})$, fix a set $B \in \mathscr{M}_{\hat{\nu}}$ and a real $\varepsilon >0$. Observe by Claim \ref{claim:identitynuhat} that $\hat{\nu}(B) \le \hat{\nu}(S)=\nu(X)<\infty$. By the definition of $\hat{\nu}$, there exists a sequence $(A_n: n \in \omega)\in \Delta(B)$ such that $\sum_n \nu(A_n) < \hat{\nu}(B)+\nicefrac{\varepsilon}{2}$. In particular, the latter series is convergent and we can pick $n_0 \in \omega$ such that $\sum_{n\ge n_0} \nu(A_n)< \nicefrac{\varepsilon}{2}$. 
    By $\sigma$-subadditivity of $\hat{\nu}$ and Claim \ref{claim:identitynuhat} we get
    $$
    \hat{\nu}\left(B_\infty\right)\le \sum_{n\in \omega} \hat{\nu}(\phi(\pi(A_n)))=\sum_{n\in \omega} \nu(A_n) < \hat{\nu}(B)+\frac{\varepsilon}{2},
    $$
    where $B_\infty:=\bigcup_{n} \phi(\pi(A_n))$. 
    In addition, since $B\in \mathscr{M}_{\hat{\nu}}$ and $B\subseteq B_\infty$, we get $\hat{\nu}(B_\infty)=\hat{\nu}(B)+\hat{\nu}(B_\infty\setminus B)$, hence $\hat{\nu}(B_\infty\setminus B)<\nicefrac{\varepsilon}{2}$. At this point, define the clopen set $C:=\bigcup_{n \in n_0}\phi(\pi(A_n))$. It follows by the above premises that 
    \begin{displaymath}
        \begin{split}
            \hat{\nu}(B\bigtriangleup C)
            &\le \hat{\nu}(B\setminus C)+\hat{\nu}(C\setminus B)\\
            &\le \hat{\nu}\left(B_\infty\setminus C\right)+\hat{\nu}\left(B_\infty\setminus B\right)
            < \sum_{n\ge n_0}\nu(A_n) + 
            \frac{\varepsilon}{2}<\varepsilon.
        \end{split}
    \end{displaymath}
    Since $\varepsilon$ is arbitrary, it follows that $B \in \mathrm{Cl}_{\hat{\nu}}(\mathcal{C})$. Therefore $\mathscr{M}_{\hat{\nu}}\subseteq \mathrm{Cl}_{\hat{\nu}}(\mathcal{C})$. 
\end{proof}

\medskip

As an application, we show that $\mathrm{Cl}_{\hat{\nu}}(\mathcal{C})$ coincides with $\mathscr{M}_{\hat{\nu}}$ if and only if $\nu$ is additive.
\begin{prop}\label{prop:characteadditivity}
    Let $\nu: \Sigma\to \overline{\mathbb{R}}$ be a submeasure such that $\nu(X)<\infty$. Then the following are equivalent\textup{:}
    \begin{enumerate}[label={\rm (\roman{*})}]
    \item \label{item:1additivity} $\nu$ is finitely additive\textup{;}
\item \label{item:2additivity} $\mathcal{C}\subseteq \mathscr{M}_{\hat{\nu}}$\textup{;}
\item \label{item:3additivity} $\mathrm{Cl}_{\hat{\nu}}(\mathcal{C})=\mathscr{M}_{\hat{\nu}}$\textup{.}
    \end{enumerate}
\end{prop}
\begin{proof}
    \ref{item:1additivity} $\implies$ \ref{item:2additivity}. Pick $C \in \mathcal{C}$ and $A \in \Sigma$ such that $C=\phi(\pi(A))$. 
    We need to show that $C \in \mathscr{M}_{\hat{\nu}}$. To this aim, 
    fix a subset $B\subseteq S$. We proceed as in \cite[Proposition 1.13(b)]{MR1681462}: for each $k \in \omega$, there exists $(A_{k,n}: n \in\omega) \in \Delta(B)$ such that $\sum_n \nu(A_{k,n}) \le \hat{\nu}(B)+2^{-k}$. By the additivity of $\nu$, 
    we obtain 
    \begin{displaymath}
        \hat{\nu}(B)+2^{-k} \ge \sum_{n\in \omega} \nu(A_{k,n}\cap A)+\sum_{n \in \omega}\nu(A_{k,n} \setminus A) \ge \hat{\nu}(B\cap C)+\hat{\nu}(B\setminus C), 
    \end{displaymath}
    where the last inequality follows by the definition of $\hat{\nu}$. 
    Since $\hat{\nu}$ is subadditive and $k$ is arbitrary, we conclude that $C \in \mathscr{M}_{\hat{\nu}}$. 

\medskip

    \ref{item:2additivity} $\implies$ \ref{item:3additivity}. Thanks to Lemma \ref{lem:secondrepresentation} and the standing hypothesis, $\mathscr{M}_{\hat{\nu}}$ is a $\sigma$-field containing both $\mathcal{N}_{\hat{\nu}}$ and $\mathcal{C}$. Hence $\mathrm{Cl}_{\hat{\nu}}(\mathcal{C})\subseteq \sigma(\mathcal{C}\cup \mathcal{N}_{\hat{\nu}})\subseteq \mathscr{M}_{\hat{\nu}}$ by Lemma \ref{lem:closureclopenlemma}. The opposite inclusion is contained also in Lemma \ref{lem:secondrepresentation}. 

\medskip

    \ref{item:3additivity} $\implies$ \ref{item:1additivity}. Pick disjoint sets $A,B \in \Sigma$. Then $\phi(\pi(A))$ and $\phi(\pi(B))$ are disjoint clopen subsets of $S$. In particular, by item \ref{item:3additivity}, the latter sets belong to $\mathscr{M}_{\hat{\nu}}$. At this point, recall that $\hat{\nu}$ is, in particular, additive on $\mathscr{M}_{\hat{\nu}}$, thanks to Lemma \ref{lem:secondrepresentation}. Therefore 
    $$
    \nu(A\cup B)=\hat{\nu}(\phi(\pi(A\cup B)))=\hat{\nu}(\phi(\pi(A)))+\hat{\nu}(\phi(\pi(B)))=\nu(A)+\nu(B),
    $$
    where we used Claim \ref{claim:identitynuhat}. 
\end{proof}

\medskip

As it follows by Example \ref{example:ClCnotsigmafield}, it is possible that $\mathrm{Cl}_{\hat{\nu}}(\mathcal{C})$ is a proper subset of $\mathcal{P}(S)$. Below, we show that same conclusion holds also in the 
additive case. 
\begin{cor}\label{cor:exoticcorollary}
\textup{(AC)} Let $\mu: \mathcal{P}(\omega)\to \mathbb{R}$ be a normalized shift-invariant finitely additive map.  
Then $\mathrm{Cl}_{\hat{\mu}}(\mathcal{C}) \neq \mathcal{P}(S)$. 
\end{cor}
\begin{proof}
    Suppose for the sake of contradiction that $\mathrm{Cl}_{\hat{\mu}}(\mathcal{C})=\mathcal{P}(S)$. We will use a Vitaly--type construction. 
    Let $\lambda: \omega\to \omega$ be the shift map $n\mapsto n+1$. This induces a homeomorphism $\tau: S\to S$. Observe that, for each clopen $B=\phi(\pi(A))$ with $A \in \Sigma$, $\tau[B]=\phi(\pi(\lambda[A])))$ and $\mu(A)=\mu(\lambda[A])$. Consequently $\hat{\mu}$ is $\tau$-invariant. 
    Now, let $\sim$ be the equivalence relation on $S$ so that $x\sim y$ if and only if $x=\tau^k(y)$ for some $k\in \mathbb{Z}$. Using the Axiom of Choice, we can construct a set $V\subseteq S$ which selects exactly one element from each equivalence class. It follows by Lemma \ref{lem:secondrepresentation} 
    that 
    $\hat{\mu}$ is $\sigma$-additive on $\mathscr{M}_{\hat{\mu}}=\mathrm{Cl}_{\hat{\mu}}(\mathcal{C})=\mathcal{P}(S)$. Taking into account that $\{\tau^k[V]: k \in \mathbb{Z}\}$ is a partition of $S$ by construction, we obtain that 
    $$
    1=\hat{\mu}(S)=\sum_{k \in \mathbb{Z}}\hat{\mu}(\tau^k[V])=\sum_{k \in \mathbb{Z}}\hat{\mu}(V),
    $$
    which is impossible for every value of $\hat{\mu}(V)$. 
\end{proof}
Examples of finitely additive maps $\mu$ satisfying the hypothesis of Corollary \ref{cor:exoticcorollary} (which are 
upper densities as in Section \ref{subsec:negative}) can be found in \cite[Remark 3]{MR4054777}. 

\medskip

Several properties of the analogous families with respect to $\tilde{\nu}$ are listed below. 
To this aim, 
given a set $B\subseteq S$, we denote by $B^\circ$ its interior and $\overline{B}$ its closure. 
\begin{lem}\label{lem:tildenulemma}
The following properties hold\text{:}
\begin{enumerate}[label={\rm (\roman{*})}]
\item \label{item:1tildenu} $\tilde{\nu}(\partial B) = 0$ for every $B \in \mathrm{Cl}_{\tilde{\nu}}(\mathcal{C})$\textup{;}

\item \label{item:2tildenu} $\partial B$, $B^{\circ}$, and $\overline{B}$ are elements of $\mathrm{Cl}_{\tilde{\nu}}(\mathcal{C})$ for every $B \in \mathrm{Cl}_{\tilde{\nu}}(\mathcal{C})$\textup{;}

\item \label{item:3tildenu} $\tilde{\nu}(\overline{B}) = \tilde{\nu}(B)$ for every $B \in \mathrm{Cl}_{\tilde{\nu}}(\mathcal{C})$\textup{;}

\item \label{item:4tildenu} $\mathcal{N}_{\tilde{\nu}} = \mathrm{nwd}(S) \cap \mathrm{Cl}_{\tilde{\nu}}(\mathcal{C})$\textup{;}

\item \label{item:5tildenu} A subset of $S$ is the boundary of some element of $\mathrm{Cl}_{\tilde{\nu}}(\mathcal{C})$ if and only if it is closed and $\tilde{\nu}$-null\textup{;}

\item \label{item:6tildenu} If $B\subseteq S$ is a compact $G_\delta$ set, then $B \in \sigma(\mathcal{C})$\textup{;}

\item \label{item:7tildenu} For every $B \in \mathrm{Cl}_{\hat{\nu}}(\mathcal{C})$, there are sets $A, C \in \sigma(\mathcal{C})$ such that $A \subseteq B \subseteq C$ and $\hat{\nu}(C \setminus A) = 0$\textup{;}

\item \label{item:8tildenu} For every $B \in \mathrm{Cl}_{\tilde{\nu}}(\mathcal{C})$, there are sets $A,C\subseteq S$ such that 
$A$ is a countable union of clopen sets, $C$ is a countable intersection of clopen sets, 
$A \subseteq B \subseteq C$, and $\tilde{\nu}(C \setminus A) = 0$\textup{.}
\end{enumerate}
\end{lem}
\begin{proof}
\ref{item:1tildenu} Fix $B \in \mathrm{Cl}_{\tilde{\nu}}(\mathcal{C})$ and $k \in \omega$. By Remark~\ref{rmk:peanocompleation}, there are $A_k, C_k \in \mathcal{C}$ with $A_k \subseteq B \subseteq C_k$ and $\tilde{\nu}(C_k \setminus A_k) \le 2^{-k}$. But since $A_k$ and $C_k$ are clopen, we have $A_k\subseteq B^\circ$ and $\overline{B}\subseteq C_k$. Hence $\partial B \subseteq C_k \setminus A_k$ and $\tilde{\nu}(\partial B) \le 2^{-k}$. The conclusion follows by the arbitrariness of $k$.

\medskip

\ref{item:2tildenu} Since any $\tilde{\nu}$-null set is the $d_{\tilde{\nu}}$-limit of a sequence of empty sets, we have 
\begin{equation}\label{eq:inclusionnutilde}
\mathcal{N}_{\tilde{\nu}}\subseteq \mathrm{Cl}_{\tilde{\nu}}(\mathcal{C}).
\end{equation} 
Now, fix $B \in \mathrm{Cl}_{\tilde{\nu}}(\mathcal{C})$. Since $\partial B \in \mathcal{N}_{\tilde{\nu}}$ by item \ref{item:1tildenu}, then $\partial B \in \mathrm{Cl}_{\tilde{\nu}}(\mathcal{C})$. Since $\mathrm{Cl}_{\tilde{\nu}}(\mathcal{C})$ is a field of sets, then also $B^\circ=B\setminus \partial B$ and $\overline{B}=B\cup \partial B$ belong to $\mathrm{Cl}_{\tilde{\nu}}(\mathcal{C})$.  

\medskip

\ref{item:3tildenu} Taking into account that $\tilde{\nu}$ is a submeasure, this is immediate by item \ref{item:1tildenu}.

\medskip

\ref{item:4tildenu} Pick $B \in \mathrm{nwd}(S) \cap \mathrm{Cl}_{\tilde{\nu}}(\mathcal{C})$. Then $\overline{B} \in \mathrm{Cl}_{\tilde{\nu}}(\mathcal{C})$ by item~\ref{item:2tildenu}. In addition $\overline{B}=\partial B$ since $B$ is nowhere dense, hence $\tilde{\nu}(B) = 0$ by item~\ref{item:1tildenu}. This shows that  $\mathrm{nwd}(S) \cap \mathrm{Cl}_{\tilde{\nu}}(\mathcal{C}) \subseteq \mathcal{N}_{\tilde{\nu}}$. 

To show the converse inclusion, pick $B \in \mathcal{N}_{\tilde{\nu}}$. Observe that $0=\tilde{\nu}(B)=\tilde{\nu}(\overline{B})$ by item~\ref{item:3tildenu}, hence $\overline{B}$ cannot contain a nonempty clopen set (since they have positive $\tilde{\nu}$-measure). Hence $B$ is nowhere dense. Taking into account also \eqref{eq:inclusionnutilde}, we obtain that $\mathcal{N}_{\tilde{\nu}}\subseteq \mathrm{nwd}(S) \cap \mathrm{Cl}_{\tilde{\nu}}(\mathcal{C})$. 

\medskip

\ref{item:5tildenu} Suppose that $B \subseteq S$ is closed and $\tilde{\nu}$-null. Thanks to item \ref{item:4tildenu}, we have $B$ is a closed set in $\mathrm{nwd}(S) \cap \mathrm{Cl}_{\tilde{\nu}}(\mathcal{C})$. Then $B$ is the boundary of the dense open set $B^c \in \mathrm{Cl}_{\tilde{\nu}}(\mathcal{C})$. This proves the \textsc{If} part. The \textsc{Only If} part follows by item \ref{item:1tildenu}.

\medskip

\ref{item:6tildenu} 
Pick a compact $G_\delta$ set $B\subseteq S$. Since $S$ is compact Hausdorff, 
it follows by \cite[Corollary 1.5.2 and Theorem 3.1.9]{MR1039321} that there exists a continuous function $f: S\to [0,1]$ such that $B=f^{-1}(0)$. 
Let $\mathscr{F}$ be the family of functions $g\in \mathbb{R}^S$ such that the image $g[S]$ is finite and $g^{-1}(x) \in \mathcal{C}$ for each $x \in \mathbb{R}$. Of course, such maps 
belong to the Banach lattice $C(S)$ of continuous functions $S\to \mathbb{R}$. 
Since $\mathscr{F}$ is dense in $C(S)$ (with respect to uniform convergence) by Stone--Weierstrass theorem, see e.g. \cite[Theorem 3.2.21]{MR1039321}, 
it follows that $f \in C(S)$ is the limit of a sequence of elements in $\mathscr{F}$. 
Taking into account the maps in $\mathscr{F}$ are $\sigma(\mathcal{C})$-measurable, we conclude that 
$f$ is $\sigma(\mathcal{C})$-measurable, hence $B \in \sigma(\mathcal{C})$. 


\medskip

\ref{item:7tildenu} Fix $B \in \mathrm{Cl}_{\hat{\nu}}(\mathcal{C})$ and $\epsilon > 0$. By definition of $\mathrm{Cl}_{\hat{\nu}}(\mathcal{C})$, there is $D \in \mathcal{C}$ and a sequence $(E_n)_{n \in \omega}$ in $\mathcal{C}$ such that $B \triangle D \subseteq \bigcup_{n} E_n$ and $\sum_{n} \hat{\nu}(E_n) < \epsilon$. Define $E := \bigcup_{n} E_n \in \sigma(\mathcal{C})$, so that 
$$
D \setminus E \subseteq B \subseteq D \cup E
\quad \text{ and }\quad 
\hat{\nu}((D \cup E) \setminus (D \setminus E)) = \hat{\nu}(E) \le \sum_{n \in \omega}\hat{\nu}(E_n)< \epsilon.
$$
By the arbitrariness of $\epsilon > 0$, we conclude that there exist $A, C \in \sigma(\mathcal{C})$ such that $A \subseteq B \subseteq C$ and $\hat{\nu}(C \setminus A) = 0$.

\medskip

\ref{item:8tildenu} Fix $B \in \mathrm{Cl}_{\tilde{\nu}}(\mathcal{C})$ and, for each $k \in \omega$, choose $A_k$ and $C_k$ as in the proof of item \ref{item:1tildenu}. 
Then $A^c := \bigcap_{k} A_k^c$ and $C := \bigcap_{k} C_k$ are 
countable intersections of clopen sets 
since $A_k,C_k \in \mathcal{C}$ for each $k \in \omega$. 
Therefore $A=\bigcup_kA_k \subseteq B \subseteq C$ and $\tilde{\nu}(C \setminus A) = 0$. 
\end{proof}

\color{black}
\medskip

We conclude with the proofs of 
Theorem \ref{thm:stonecharacterization} and Corollary \ref{cor:corollaryjoncharacterization}. 

\begin{proof}
    [Proof of Theorem \ref{thm:stonecharacterization}] 
    Let us start defining also the following condition: 
    \begin{enumerate}[label={\rm (\textsc{a}\arabic*$^\prime$)}, itemsep=1mm]
    \item \label{item:B1primestone} For every increasing sequence $(A_n: n \in\omega)$ in $\overline{\Sigma}$ such that $\sum_n \nu^\star(A_{n+1}\setminus A_n)<\infty$, there exists $A \in \overline{\Sigma}$ such that $\nu^\star(A_n\setminus A)=0$ for all $n \in \omega$ and $\lim_n \nu^\star(A\setminus A_n)=0$\textup{.} 
    \end{enumerate}
    Thus, it will be enough to prove all the implications depicted in Figure \ref{fig:figure1} below 
    (here, the implication 
    \ref{item:B3stone} $\implies$ \ref{item:B1primestone} 
    may seem superfluous, but it is a step in the proof of 
    \ref{item:C3stone} $\implies$ \ref{item:D1stone}
    ).
    
\begin{figure}[!htbp]
\centering
\begin{tikzpicture}
[xscale =1.2, yscale=1.2] 
\node (a) at (0-.1,0.9){\ref{item:B1stone}};

\node (b1) at (0-2,1*1.3-.4){\ref{item:B1primestone}};
\node (b2) at (-.866*1.3-2,-.5*1.3-.4){\ref{item:B2stone}};
\node (b3) at (.866*1.3-2,-.5*1.3-.4){\ref{item:B3stone}};
\draw[thin,<->, double, >=stealth](b1)--(a);
\draw[thin,->, double, >=stealth](b1)--(b2);
\draw[thin,->, double, >=stealth](b3)--(b1);
\draw[thin,->, double, >=stealth](b2)--(b3);



\node (c1) at (0+6,1*1.3-.4){\ref{item:C1stone}};
\node (c1prime) at (8.4,1*1.3-.4){\ref{item:C1ABCD}};
\node (c2) at (-.866*1.3+6,-.5*1.3-.4){\ref{item:C3stone}};
\node (c3) at (.866*1.3+6,-.5*1.3-.4){\ref{item:C2stone}};
\draw[thin,->, double, >=stealth](c1)--(c1prime);
\draw[thin,->, double, >=stealth](c1prime)--(c3);
\draw[thin,->, double, >=stealth](c3)--(c2);
\draw[thin,->, double, >=stealth](c2)--(c1);


\node (d1) at (0+3-.2,1*1.3-.4+2){\ref{item:D2stone}};
\node (d2) at (-.866*1.3+3-.2,-.5*1.3-.4+2){\ref{item:D3stone}};
\node (d3) at (.866*1.3+3-.2,-.5*1.3-.4+2){\ref{item:D1stone}};
\draw[thin,->, double, >=stealth](d3)--(d1);
\draw[thin,->, double, >=stealth](d1)--(d2);
\draw[thin,->, double, >=stealth](d2)--(a);

\draw[thin,->, double, >=stealth](b3)--(c2);
\draw[thin,->, double, >=stealth](c2)--(d3);

\node (ddd1) at (3.9,-2.95){\ref{item:DDD1stone}};
\node (ddd2) at (1.7,-2.95){\ref{item:DDD7stone}}; 
\node (ddd3) at (2.8-2.4+.4,-4.95){\ref{item:DDD3stone}};
\node (ddd4) at (2.8+2.4-2.4,-4.95){\ref{item:DDD4stone}};
\draw[thin,->, double, >=stealth](ddd1)--(ddd2);
\draw[thin,->, double, >=stealth](ddd4)--(ddd3);
\draw[thin,->, double, >=stealth](c2)--(ddd1);
\draw[thin,->, double, >=stealth](ddd1)--(ddd4);

\node (ddd2B) at (6.1,-2.95){\ref{item:DDD2stone}};
\draw[thin,<->, double, >=stealth](ddd1)--(ddd2B);

\node (e1) at (.9-2,-4.95){\ref{item:E1stone}}; 
\node (ddd7) at (1.7,-1.95){\ref{item:DDD3Bstone}};
\node (e2) at (-.1,-2.95){\ref{item:E2stone}};
\node (ddd3B) at (1.7,-3.95){\ref{item:DDD9stone}};

\node (e3) at (-2,-2.95){\ref{item:E3stone}};
\draw[thin,->, double, >=stealth](ddd3)--(e1);
\draw[thin,->, double, >=stealth](ddd2)--(e2);
\draw[thin,->, double, >=stealth](e1)--(e2);
\draw[thin,->, double, >=stealth](e2)--(e3);
\draw[thin,->, double, >=stealth](e3)--(b2);

\draw[thin,->, double, >=stealth](ddd1)--(ddd3B);
\draw[thin,->, double, >=stealth](ddd1)--(ddd7);
\draw[thin,->, double, >=stealth](ddd3B)--(e2);
\draw[thin,->, double, >=stealth](ddd7)--(e2);

\end{tikzpicture}
\caption{Implications in the proof of Theorem \ref{thm:stonecharacterization}.}
\label{fig:figure1}
\end{figure}

    \medskip 
    
    \ref{item:B1stone} $\Longleftrightarrow$ \ref{item:B1primestone}. It follows by Theorem \ref{thm:joncharacterization}. 

   \medskip
   
    \ref{item:B1primestone} $\implies$ \ref{item:B2stone}. 
    Pick an increasing sequence $(A_n: n \in\omega)$ in $\Sigma$ such that $\sum_n \nu(A_{n+1}\setminus A_n)<\infty$, and fix $A\in \overline{\Sigma}$ as in item \ref{item:B1primestone}. 
    By the definition of $\nu^\star$, for each $m \in \omega$ there exists $E_m \in \Sigma$ such that $A\setminus A_m \subseteq E_m$ and $\nu^\star(E_m) \le \nu^\star(A\setminus A_m)+2^{-m}$. 
    Define the decreasing sequence $(D_m: m \in \omega)$ with values in $\Sigma$ by $D_m:=\bigcap_{k\le m}(A_k \cup E_k)$ for each $m \in \omega$. Since $A\subseteq A_k\cup E_k$ for each $k \in \omega$, we have that $A\subseteq D_m$ for all $m \in \omega$. This implies that $\nu(A_n\setminus D_m) \le \nu^\star(A_n\setminus A)=0$ for all $n,m \in \omega$. Lastly, since
    $$
    \nu(D_n\setminus A_n) \le \nu((A_n\cup E_n)\setminus A_n)\le \nu(E_n) \le \nu^\star(A\setminus A_n)+2^{-n}
    $$
    for all $n \in \omega$, it follows that $\lim_n \nu(D_n\setminus A_n)=0$. 
    
    \medskip

    \ref{item:B2stone} $\implies$ \ref{item:B3stone}. 
    Pick an increasing sequence $(A_n: n \in\omega)$ in $\Sigma$ such that $\sum_n \nu(A_{n+1}\setminus A_n)<\infty$. By item \ref{item:B2stone}, there exists $(D_m: m \in \omega)$ in $\Sigma$ such that $\nu(A_n\setminus D_m)=0$ for all $n,m \in \omega$ and $\lim_n \nu(D_n\setminus A_n)=0$. We claim that 
    \begin{equation}\label{eq:definitionCimplicationstone}
    C:=\bigcup_{k \in \omega} (A_k\cap D_k)
    \end{equation}
    satisfies the required property in item \ref{item:B3stone}. Since $A_n\setminus C\subseteq A_n\setminus D_n$, it follows by monotonicity that $\nu^\star(A_n\setminus C)=0$ for each $n \in \omega$. Moreover, since $C\subseteq A_m\cup D_m$ for all $m \in \omega$ and $\nu^\star$ is a submeasure, we get that $\nu^\star(C\setminus A_n) \le \nu((A_n\cup D_n)\setminus A_n)=\nu(D_n\setminus A_n)$, which converges to $0$ as $n\to \infty$. 

\medskip

    \ref{item:B3stone} $\implies$ \ref{item:B1primestone}. 
    Pick an increasing $(A_n: n \in \omega)$ in $\overline{\Sigma}$ such that $\sum_n \nu^\star(A_{n+1}\setminus A_n)<\infty$. 
    For each $n\in \omega$, there exists $B_n\in \Sigma$ such that $\nu^\star(A_n \bigtriangleup B_n) < 2^{-n}$. Then there is also $C_n\in \Sigma$ such that $A_n \bigtriangleup B_n \subseteq C_n$ and $\nu(C_n) < 2^{-n}$. At this point, set $D_n:=A_n\setminus C_n$ for each $n \in \omega$ and note that $D_n\subseteq A_n$ and $D_n=B_n\setminus C_n \in \Sigma$. Hence the sequence $(E_n: n \in \omega)$ defined by $E_n:=\bigcup_{k \in n+1}D_k$ is increasing with values in  $\Sigma$. 
    Now, observing that $E_n \subseteq A_n$ and $A_n\setminus E_n\subseteq C_n$, we have 
    \begin{displaymath}
        \begin{split}
            \nu(E_{n+1} \setminus E_n) &
    \le \nu^\star(E_{n+1} \setminus A_{n+1}) + \nu^\star(A_{n+1} \setminus A_n) + \nu^\star(A_n \setminus E_n) \\
    &\le \nu^\star(A_{n+1} \setminus A_n) + \nu^\star(C_n)
    < \nu^\star(A_{n+1} \setminus A_n) + 2^{-n}.
        \end{split}
    \end{displaymath}
    Hence $\sum_n \nu(E_{n+1} \setminus E_n)<\infty$. It follows, thanks to item \ref{item:B3stone}, that there exists $A\subseteq X$ such that $\nu^\star(E_n\setminus A)=0$ for all $n \in \omega$ and $\lim_n \nu^\star(A\setminus E_n)=0$.  

    To complete the implication, it is enough to show that the set $A$ satisfies the property given in item \ref{item:B1primestone}. First, observe that 
    $$
    \limsup_{n\to \infty} d_{\nu^\star}(E_n,A)\le \limsup_{n\to \infty}\, (\nu^\star(E_n\setminus A)+\nu^\star(A\setminus E_n))=0.
    $$ 
    Therefore $A \in \overline{\Sigma}$. 
    Moreover, since $A_n \triangle E_n = A_n \setminus E_n \subseteq A_n \setminus D_n \subseteq C_n$ for all $n \in \omega$, we have 
    \begin{displaymath}
        \begin{split}
            \limsup_{n\to \infty} \nu^\star(A_n \triangle A) &\leq 
    \limsup_{n\to \infty} \nu^\star(A_n \triangle E_n) + \limsup_{n\to \infty} \nu^\star(E_n \triangle A)\\
    &\leq 
    \limsup_{n\to \infty} \nu(C_n) + \limsup_{n\to \infty} \nu^\star(E_n \triangle A) = 0. 
        \end{split}
    \end{displaymath}
This implies that $\lim_n \nu^\star(A \setminus A_n)  =0$ and, since $(A_n: n \in \omega)$ is increasing, also $\nu^\star(A_n\setminus A) \le \limsup_k \nu^\star(A_k\setminus A)=0$ for each $n \in \omega$.

    


    
    
\medskip

\ref{item:B3stone} $\implies$ \ref{item:C3stone}. Pick $(A_n: n \in \omega)$ as in item \ref{item:C3stone}. Hence by \ref{item:B3stone} there exists $C\subseteq X$ such that $\nu^\star(A_n\setminus C)=0$ for all $n \in \omega$ and $\lim_n \nu^\star(C\setminus A_n)=0$. Now, observe that 
\begin{equation}\label{eq:upperboundnustarAn}
\forall n \in \omega, \quad 
\nu(A_n) \le \nu^\star(A_n\cup C) \le \nu^\star(C) + \nu^\star(A_{n}\setminus C)=\nu^\star(C). 
\end{equation}
In addition, we have that $\nu^\star(C)\le \nu^\star(C\setminus A_n)+\nu^\star(A_n\cap C)$ for each $n \in \omega$. Together with $\nu^\star(A_n\setminus C)=0$ and Inequality \eqref{eq:upperboundnustarAn}, we get
$$
0\le \nu^\star(C)-\nu^\star(A_n) \le \nu^\star(C\setminus A_n)+\nu^\star(A_n\cap C)-\nu^\star(A_n)=\nu^\star(C\setminus A_n)
$$
for each $n \in \omega$. Considering that $\lim_n \nu^\star(C\setminus A_n)=0$ and $(A_n: n \in \omega)$ is increasing, we conclude that $\nu^\star(C)=\sup_n \nu(A_n)$. 

\medskip

\ref{item:C1stone} 
$\implies$ \ref{item:C1ABCD}. 
This is clear. 

\medskip

\ref{item:C1ABCD} $\implies$ \ref{item:C2stone}. Pick an increasing sequence $(A_n: n \in \omega)$ in $\Sigma$ such that $A_0=\emptyset$ and $\sum_n \nu(A_{n+1}\setminus A_n)<\infty$, and define the countable union of clopen sets
$$
U:=\bigcup_{n\in \omega} \phi(\pi(A_n)). 
$$
Considering that $(A_{n+1}\setminus A_n: n \in\omega) \in \Delta(U)$, it follows 
that 
\begin{equation}\label{eq:kjhsjgf}
\hat{\nu}(\overline{U})=\hat{\nu}(U) 
\le 
\sum_{n\in\omega} \nu(A_{n+1}\setminus A_n)<\infty. 
\end{equation}

By the definition of $\hat{\nu}$, for each $k \in \omega$ there exists a sequence $(C_{k,n}: n \in \omega)\in \Delta(\overline{U})$ such that $\sum_n \nu(C_{k,n}) \le \hat{\nu}(\overline{U})+2^{-k}$. Taking into account that $(\phi(\pi(C_{k,n})):n\in \omega)$ is an open cover of the compact set $\overline{U}$, there exists $I_k \in \mathrm{Fin}$ such that $\overline{U}\subseteq \bigcup_{n \in I_k}\phi(\pi(C_{k,n}))$. 
Thus, it is sufficient to show that the decreasing sequence $(D_m: m \in \omega)$ given by
$$
\forall m \in \omega, \quad 
D_m:=\bigcap_{k\in m+1} \bigcup_{n \in I_k}C_{k,n}
$$
satisfies the required properties. 

To this aim, note that $D_m \in \Sigma$ and $\overline{U}\subseteq \phi(\pi(D_m))$ for each $m \in \omega$. In addition, for each $n,m \in \omega$, we have $\phi(\pi(A_n))\setminus \phi(\pi(D_m)) \subseteq U\setminus \overline{U}=\emptyset$, hence $\pi(A_n) \le \pi(D_m)$, which implies $\nu(A_n\setminus D_m)=0$. Lastly, for each $k \in\omega$, using also Claim \ref{claim:identitynuhat} and the fact that $(A_k, A_{k+1}\setminus A_k, A_{k+2}\setminus A_{k+1}, \ldots) \in \Delta(U)$, we obtain 
\begin{displaymath}
    \begin{split}
        0\le \nu(D_k)-\nu(A_k) &=\hat{\nu}(\phi(\pi(D_k))-\nu(A_k) \\
        &\le \hat{\nu}(\overline{U})-\nu(A_k)+2^{-k}\\
        &=\hat{\nu}(U)-\nu(A_k)+2^{-k}
        \le \sum_{n\ge k} \nu(A_{n+1}\setminus A_n)+2^{-k}.
    \end{split}
\end{displaymath}
The implication follows by the fact the right hand side converges to $0$ as $k\to \infty$.

\medskip

\ref{item:C2stone} $\implies$ \ref{item:C3stone}. Pick $(A_n: n \in \omega)$ and $(D_m: m \in \omega)$ as in item \ref{item:C2stone}, and define 
$
C:=\bigcup_{k
}(A_k \cap D_k).
$ 
Since $A_n\setminus C\subseteq A_n\setminus D_n$, it follows by monotonicity that $\nu^\star(A_n\setminus C)=0$ for each $n \in \omega$. Now, observe that by Inequality \eqref{eq:upperboundnustarAn} we get  $\sup_n \nu(A_n) \le \nu^\star(C)$. Conversely, since $C\subseteq A_m\cup D_m$ for all $m \in \omega$ and $\nu^\star$ is a submeasure, we get that $\nu^\star(C) \le \nu(A_m\setminus D_m)+\nu(D_m)=\nu(D_m)$. Hence by item \ref{item:C2stone} we conclude that 
$\nu^\star(C) \le \inf_m \nu(D_m)=\sup_n \nu(A_n)$.

\medskip
    
    \ref{item:C3stone} $\implies$ \ref{item:C1stone}. Fix a set $U\subseteq S$. Since $\hat{\nu}$ is a capacity, it is obvious that $\hat{\nu}(\overline{U})=\infty$ if $\hat{\nu}(U)=\infty$. Hence, let us suppose hereafter that $\hat{\nu}(U)<\infty$. 
    
    Fix $k \in \omega$. By the definition of $\hat{\nu}$, there exists  $(F_{k,n}: n \in \omega) \in \Delta(U)$ such that 
    \begin{equation}\label{eq:firstupperboundFkn}
    \sum_{n \in \omega}\nu(F_{k,n})\le \hat{\nu}(U) +2^{-k}. 
    \end{equation}
    Assume without loss of generality that $F_{k,0}=\emptyset$. For each $n \in \omega$, define $A_{k,n}:=F_{k,0} \cup \cdots \cup F_{k,n}$, so that $(A_{k,n}: n \in \omega)$ is increasing in $\Sigma$ and $\sum_n \nu(A_{k,n+1}\setminus A_{k,n})\le \sum_{n}\nu(F_{k,n})<\infty$, where the latter follows by \eqref{eq:firstupperboundFkn}. By item \ref{item:C3stone}, there exists $C_k\subseteq X$ such that $\nu^\star(A_{k,n}\setminus C_k)=0$ for all $n \in\omega$ and 
    $\sup_n \nu(A_{k,n})=\nu^\star(C_k)$. 
    In particular, there exists $n_k \in \omega$ such that $\nu^\star(C_k) \le \nu(A_{k,n_k})+2^{-k}$. 
    In addition, by the definition of $\nu^\star$, there exists $D_k \in \Sigma$ such that $C_k\subseteq D_k$ and $\nu(D_k) \le \nu^\star(C_k)+2^{-k}$. 
    
    \begin{claim}\label{claim:upperboundCk}
    $\nu^\star(C_k) \le \hat{\nu}(U)+2^{1-k}$. 
    \end{claim}
    \begin{proof}
    Since $A_{k,n_k}=\bigcup_{n \in n_k}A_{k,n+1}\setminus A_{k,n}$ and $\nu$ is a submeasure, we obtain that 
        $$
        \nu(A_{k,n_k})\le \sum_{n \in \omega}\nu(A_{k,n+1}\setminus A_{k,n}) \le \sum_{n \in \omega}\nu(F_{k,n}). 
        $$
        The claim follows taking into account also Inequality \eqref{eq:firstupperboundFkn} and that $\nu^\star(C_k) \le \nu(A_{k,n_k})+2^{-k}$. 
    \end{proof}

    \begin{claim}\label{claim:closure}
    $\overline{U}\subseteq \phi(\pi(D_k))$.
    \end{claim}
    \begin{proof}
    Observe, for each $n \in \omega$, that $\nu(A_{k,n}\setminus D_k)\le \nu^\star(A_{k,n}\setminus C_k)=0$ which implies that  $\pi(A_{k,n})\le \pi(D_k)$, hence $\phi(\pi(A_{k,n}))\subseteq \phi(\pi(D_k))$. Since $(F_{k,n}: n \in\omega) \in \Delta(U)$, we obtain 
         $$
         U\subseteq \bigcup_{n \in \omega} \phi(\pi(F_{k,n})) 
         \subseteq \bigcup_{n \in \omega} \phi(\pi(A_{k,n}))
         \subseteq \phi(\pi(D_{k})).
         $$
         The claim follows by the fact that the set on the right hand side is closed in $S$. 
    \end{proof}

It follows by Claim \ref{claim:identitynuhat}, Claim \ref{claim:upperboundCk}, Claim \ref{claim:closure}, 
and the above observations that 
\begin{displaymath}
    \begin{split}
        \nu^\star(C_k) &
        \le \hat{\nu}(U)+2^{1-k}
        \le \hat{\nu}(\overline{U})+2^{1-k}
        \\&
        \le \hat{\nu}(\phi(\pi(D_k)))+2^{1-k}
        = \nu(D_k)+2^{1-k}
        \le \nu^\star(C_k)+2^{2-k}.
    \end{split}
\end{displaymath}
By letting $k\to \infty$, it follows that $\hat{\nu}(U)=\hat{\nu}(\overline{U})$. 

\medskip

\ref{item:C3stone} $\implies$ \ref{item:D1stone}. Pick an increasing sequence $(A_n: n \in \omega)$ of subsets of $X$ such that $\sum_{n} \nu^\star(A_{n+1} \setminus A_n) < \infty$. Then $(A_n: n \in \omega)$ is a Cauchy sequence and in $(\mathcal{P}(X), d_{\nu^\star})$ and we can suppose without loss of generality that $\nu^\star(A_{n+1} \setminus A_n) < 2^{-n}$ for all $n \in \omega$. By the definition of $\nu^\star$, for each $n \in \omega$ there exists $V_n \in \Sigma$ such that $A_{n+1} \setminus A_n \subseteq V_n$ and $\nu(V_n) < 2^{-n}$. 
Fix $k \in \omega$ and define $(W_{k,n}: n \in \omega)$ by 
$$
\forall n \in \omega, \quad 
W_{k,n}:=V_k\cup V_{k+1}\cup \cdots \cup V_{k+n}.
$$
Of course, $W_{k,n} \in \Sigma$ and $\sum_n \nu(W_{k,n+1}\setminus W_{k,n})\le \sum_{n} \nu(V_n)<\infty$. It follows by item \ref{item:C3stone} that there exists $C_k\subseteq X$ such that $\nu^\star(W_{k,n}\setminus C_k)=0$ for all $n \in \omega$ and $\nu^\star(C_k)=\sup_n\nu(W_{k,n})\le \sum_{n\ge k}\nu(V_n)<2^{1-k}$. 

\begin{claim}\label{claim:measurezero}
$\nu^\star(A_i\setminus (A_j\cup C_j))=0$ for all $i,j \in \omega$. 
\end{claim}
\begin{proof}
    Fix $i,j \in \omega$. If $i\le j$ the claim is clear since $A_i\subseteq A_j$. Hence, suppose hereafter that $i>j$. Then 
    $$
    A_i\setminus (A_j\cup C_j) \subseteq \bigcup_{k=j}^{i-1} (A_{k+1}\setminus A_k)\setminus C_j \subseteq \bigcup_{k=j}^{i-1} V_k\setminus C_j
    \subseteq \bigcup_{n\in i-j} W_{j,n}\setminus C_j.
    $$
    The claim follows since $\nu^\star$ is a submeasure and $\nu^\star(W_{j,n}\setminus C_j)=0$ for all $n \in \omega$. 
\end{proof}

At this point, consider the decreasing sequence $(D_m: m \in \omega)$ of subsets of $X$ defined by $D_m:=\bigcap_{i \in m+1}(A_i\cup C_i)$. Since $D_m\subseteq A_m\cup C_m$, we get by Claim \ref{claim:measurezero} 
$$
\nu^\star(A_n\setminus D_m)\le \sum_{i \in m+1}\nu^\star(A_n\setminus (A_i\cup C_i))=0
$$
for all $n,m \in \omega$. Lastly, we have also 
$$
\nu^\star(D_n\setminus A_n)\le \nu^\star((A_n\cup C_n)\setminus A_n) \le \nu^\star(C_n) < 2^{1-n}
$$
for all $n \in \omega$. Therefore $\lim_n \nu^\star(D_n\setminus A_n)=0$, and the pseudometric space $(\mathcal{P}(X), d_{\nu^\star})$ satisfies its corresponding item \ref{item:B2stone}. Since we already proved that \ref{item:B1stone} $\Longleftrightarrow$ \ref{item:B2stone} then $(\mathcal{P}(X), d_{\nu^\star})$ is complete. 

\medskip

\ref{item:D2stone} $\implies$ \ref{item:D3stone}. This is clear. 

\medskip

\ref{item:D1stone} $\implies$ \ref{item:D2stone} and 
\ref{item:D3stone} $\implies$ \ref{item:B1stone}. They follow by the fact that a closed subset of a complete pseudometric space is a complete subspace.

\medskip

\ref{item:C3stone} $\implies$ \ref{item:DDD1stone}. Pick a set $U\subseteq S$. Since $\hat{\nu}\le \tilde{\nu}$, we can suppose hereafter that $\hat{\nu}(U)<\infty$. Proceeding verbatim and with the same notations as in the proof of the implication \ref{item:C3stone} $\implies$ \ref{item:C1stone}, for each $k \in \omega$ it is possible to pick a set $D_k \in \Sigma$ such that $U\subseteq \overline{U}\subseteq \phi(\pi(D_k))$ and $\nu(D_k)
\le \hat{\nu}(U)+2^{2-k}$. It follows that $\tilde{\nu}(U) \le \inf_k \nu(D_k)\le \hat{\nu}(U)$, which proves the converse inequality $\tilde{\nu} \le \hat{\nu}$. 


\medskip

\ref{item:DDD1stone} $\Longleftrightarrow$ \ref{item:DDD2stone}. The equivalence follows by the facts that $\hat{\nu}\le \tilde{\nu}$, $\hat{\nu}$ is $\sigma$-subadditive, and, by definition, it is also the maximum $\sigma$-subadditive (and monotone) extension of the submeasure $\phi(\pi(A))\mapsto \nu(A)$ to $\mathcal{P}(S)$, cf. Claim \ref{claim:identitynuhat}, whereas $\tilde{\nu}$ is the maximum monotone (and subadditive) extension of 
the same submeasure. 

%

\medskip

\ref{item:DDD1stone} $\implies$ \ref{item:DDD4stone}. This is obvious.

\medskip

\ref{item:DDD4stone} 
$\implies$
\ref{item:DDD3stone}. 
Since $\hat{\nu}\le \tilde{\nu}$, we have $\mathrm{Cl}_{\tilde{\nu}}(\mathcal{C})\subseteq \mathrm{Cl}_{\hat{\nu}}(\mathcal{C})$. Conversely, it follows by Remark \ref{rmk:peanocompleation} that item \ref{item:DDD3stone} holds if and only if $\mathrm{Cl}_{\hat{\nu}}(\mathcal{C})\subseteq \mathrm{Cl}_{\tilde{\nu}}(\mathcal{C})$.





\medskip

\ref{item:DDD3stone} $\implies$ \ref{item:E1stone}. 
Pick $B \in \mathrm{Cl}_{\hat{\nu}}(\mathcal{C})$. By hypothesis, for each $k \in \omega$, there exist $A_k, C_k \in \mathcal{C}$ such that $A_k \subseteq B \subseteq C_k$ and $\tilde{\nu}(C_k \setminus A_k) < 2^{-k}$. Since $A_k$ and $C_k$ are clopen, the interior of $B$ contains $A_k$ and the closure of $B$ is contained in $C_k$. Hence $\partial B \subseteq C_k \setminus A_k$. It follows that $\tilde{\nu}(\partial B) \le \inf_k \tilde{\nu}(C_k\setminus A_k)=0$.

\medskip

\ref{item:E1stone} $\implies$ \ref{item:E2stone}. 
This is clear since $\hat{\nu}\le \tilde{\nu}$.

\medskip

\ref{item:DDD1stone} $\implies$ \ref{item:DDD3Bstone}. 
This follows by Lemma~\ref{lem:tildenulemma}\ref{item:8tildenu}.

\medskip

\ref{item:DDD3Bstone} $\implies$ \ref{item:E2stone}. 
Fix $B \in \mathrm{Cl}_{\hat{\nu}}(\mathcal{C})$. By hypothesis, it is possible to pick sets $A,B\subseteq S$ such that $A$ is open, $C$ is closed, $A\subseteq B\subseteq C$ and $\hat{\nu}(C\setminus A)=0$. It follows that $A\subseteq B^\circ$ and $\overline{B}\subseteq C$. Since $\partial B\subseteq \overline{B}\setminus B^\circ\subseteq C\setminus A$, we conclude that $\hat{\nu}(\partial B)\le \hat{\nu}(C\setminus A)=0$.

\medskip

\ref{item:DDD1stone} $\implies$ \ref{item:DDD7stone}. 
This follows by Lemma~\ref{lem:tildenulemma}\ref{item:5tildenu}.

\medskip

\ref{item:DDD1stone} $\implies$ \ref{item:DDD9stone}. It follows by 
Lemma \ref{lem:tildenulemma}\ref{item:2tildenu} 
and 
Lemma \ref{lem:tildenulemma}\ref{item:4tildenu}.

\medskip

\ref{item:DDD9stone} $\implies$ \ref{item:E2stone}. 
This follows by the fact that every $\partial B$ is nowhere dense.

\medskip

\ref{item:DDD7stone} $\implies$ \ref{item:E2stone} $\implies$ \ref{item:E3stone}. 
They are obvious.

\medskip

\ref{item:E3stone} $\implies$ \ref{item:B2stone}. Pick an increasing sequence $(A_n: n \in \omega)$ with values in $\Sigma$ such that $A_0=\emptyset$ and $\sum_n \nu(A_{n+1}\setminus A_n)<\infty$, and 
define 
$$
U:=\bigcup_{n\in \omega} \phi(\pi(A_n)). 
$$
Then $U$ is a countable union of clopen sets. 
Hence by item \ref{item:E3stone} we get $\hat{\nu}(\partial U)
=0$. 
Moreover, $\hat{\nu}(U) < \infty$, as in \eqref{eq:kjhsjgf}.  
By the definition of $\hat{\nu}$, for each $k \in \omega$ there exists a sequence $(G_{k,n}: n \in\omega) \in \Delta(\overline{U}\setminus \phi(\pi(A_k)))$ such that $\sum_n \nu(G_{k,n}) \le \hat{\nu}(\overline{U}\setminus \phi(\pi(A_k)))+2^{-k}$. 
Taking into account that $(\phi(\pi(G_{k,n})): n \in \omega)$ is an open cover of the compact set $\overline{U}\setminus \phi(\pi(A_k))$, there exists $J_k \in \mathrm{Fin}$ such that $\overline{U}\setminus \phi(\pi(A_k)) \subseteq \bigcup_{n \in J_k} \phi(\pi(G_{k,n}))$. 
Now, for each $k \in \omega$, define the set 
$$
V_k:=U\cup \bigcup_{n \in J_k}\phi(\pi(G_{k,n}))
$$
Considering that $\{\phi(\pi(G_{k,n})): n \in J_k\}$ is also a cover of $\partial U$, it follows that $V_k=\overline{U} \cup \bigcup_{n \in J_k}\phi(\pi(G_{k,n}))$. Therefore each $V_k$ is a clopen subset of $S$ containing $\overline{U}$. 
Replacing, if necessary, $V_k$ with $\bigcap_{j \in k+1}V_j$, we can also assume without loss of generality that $(V_k: k \in \omega)$ is decreasing. 
Also, since $V_k \in \mathcal{C}$ and $\phi: \Sigma/\nu \to \mathcal{C}$ is a Boolean isomorphism, there exists a decreasing sequence $(D_k: k \in \omega)$ in $\Sigma$ such that 
$$
\forall k \in \omega, \quad 
V_k=\phi(\pi(D_k)). 
$$

We claim that the above sequence satisfies the required properties. In fact, for each $n,m \in \omega$, we have $\phi(\pi(A_n))\setminus \phi(\pi(D_m)) =\phi(\pi(A_n))\setminus V_m \subseteq U\setminus \overline{U}=\emptyset$, hence $\pi(A_n) \le \pi(D_m)$, which implies $\nu(A_n\setminus D_m)=0$. Lastly, using also Claim \ref{claim:identitynuhat}, for each $k \in \omega$ we get 
\begin{displaymath}
    \begin{split}
        \nu(D_k\setminus A_k)
        &=\hat{\nu}(\phi(\pi(D_k))\setminus \phi(\pi(A_k)))\\
        &=\hat{\nu}\left(U\cup \bigcup_{n \in J_k}\phi(\pi(G_{k,n})) \setminus \phi(\pi(A_k))\right)\le \hat{\nu}(\phi(\pi(\bigcup_{n \in J_k}G_{k,n})))\\
        &\le \sum_{n \in J_k}\nu(G_{k,n}) 
        \le \sum_{n \in \omega}\nu(G_{k,n}) 
        \le \hat{\nu}(\overline{U}\setminus \phi(\pi(A_k)))+2^{-k}\\
        &\le \hat{\nu}(U\setminus \phi(\pi(A_k)))+\hat{\nu}(\partial U) +2^{-k} = \hat{\nu}\left(\bigcup_{n\ge k}\phi(\pi(A_{n+1}\setminus A_n))\right)+2^{-k}\\
        &\le \sum_{n\ge k}\nu(A_{n+1}\setminus A_n)+2^{-k}.
    \end{split}
\end{displaymath}
Therefore $\lim_k \nu(D_k\setminus A_k)=0$.
\end{proof}

\medskip

\begin{proof}
    [Proof of Corollary \ref{cor:corollaryjoncharacterization}]
    \ref{item:cor1jon} $\implies$ \ref{item:cor2jon}. Suppose that $(\Sigma, d_{\nu})$ is complete, and fix $B \in \mathrm{Cl}_{\hat{\nu}}(\mathcal{C})$. By the equivalence \ref{item:B1stone} $\Longleftrightarrow$ \ref{item:DDD3stone} in Theorem \ref{thm:stonecharacterization}, for each $n \in \omega$, there exist $A_n,C_n \in \Sigma$ such that $\phi(\pi(A_n))\subseteq B\subseteq \phi(\pi(C_n))$ and $\tilde{\nu}(\phi(\pi(C_n\setminus A_n)))=\nu(C_n\setminus A_n)<2^{-n}$. 
    Hence, $\lim_n \tilde{\nu}(\phi(\pi(A_n))\bigtriangleup B)=0$. 
    Without loss of generality, it is possible to assume that $(A_n: n \in\omega)$ is increasing and $(C_n: n \in \omega)$ is decreasing. It follows that 
    $$
    \nu(A_{m}\setminus A_n) 
    \le \sum_{k \in \omega} \nu(A_{n+k+1}\setminus A_{n+k})
    \le \sum_{k \in \omega} \nu(C_{n+k}\setminus A_{n+k})<2^{1-n}
    $$
    for all $n,m \in \omega$ with $m\ge n$, hence $(A_n: n \in \omega)$ is $d_\nu$-Cauchy. Thus it is convergent to some $A \in \Sigma$, which can be rewritten as $\lim_n \tilde{\nu}(\phi(\pi(A_n))\bigtriangleup \phi(\pi(A)))=0$. Therefore $\tilde{\nu}(\phi(\pi(A))\bigtriangleup B)=0$. 

    \medskip

    \ref{item:cor2jon} $\implies$ \ref{item:cor2Bjon}. This is clear since $\hat{\nu}\le \tilde{\nu}$.
    
    \medskip

    \ref{item:cor2Bjon} $\implies$ \ref{item:cor3jon}. 
    This is immediate.


    \medskip

    \ref{item:cor3jon} $\implies$ \ref{item:cor1jon}. 
    Pick an increasing sequence $(A_n)_{n \in \omega}$ with values in $\Sigma$ such that $\sum_{n} \nu(A_{n+1} \setminus A_n) < \infty$. 
    Define $B_n := \phi(\pi(A_n)) \in \mathcal{C}$ for all $n \in \omega$ and note that $(B_n: n \in \omega)$ converges to $B := \bigcup_{n} B_n$ in $(\mathcal{P}(S), d_{\hat{\nu}})$ by Claim \ref{claim:identitynuhat} and the $\sigma$-subadditivity of $\hat{\nu}$. 
    Hence $B$ is a countable union of clopen sets, and 
    belongs to $\mathrm{Cl}_{\hat{\nu}}(\mathcal{C})$.  
    By item \ref{item:cor3jon} it is possible to pick $A \in \Sigma$ such that 
    $
    \hat{\nu}(\phi(\pi(A)) \triangle B) = 0.
    $ 
    It follows by Claim~\ref{claim:identitynuhat} that $(A_n: n \in \omega)$ is $d_\nu$-convergent to $A$. Thanks to Theorem~\ref{thm:joncharacterization}, the pseudometric space $(\Sigma, d_{\nu})$ is complete.
\end{proof}

Finally, in Remark \ref{rmk:necessarycompleteness} and Remark \ref{rmk:sufficientcompleteness} below we provide a necessary condition and a sufficient condition for the completeness of $(\overline{\Sigma}, d_{\nu^\star})$, respectively. However, in both cases, the implications cannot be reversed. 
\begin{rmk}\label{rmk:necessarycompleteness}
By the equivalence \ref{item:B1stone} $\Longleftrightarrow$ \ref{item:DDD9stone} of Theorem \ref{thm:stonecharacterization}, 
the following implication holds: 
\begin{equation}\label{eq:jonclaimedinclusion}
(\overline{\Sigma}, d_{\nu^\star}) \text{ complete} 
\quad \implies \quad 
\mathcal{N}_{\hat{\nu}}\subseteq \mathrm{nwd}(S).
\end{equation}

However, 
the converse implication of \eqref{eq:jonclaimedinclusion} fails in general. 
To this aim, let  
$(a_n: n \in \omega)$ be a strictly positive sequence of reals such that $\sum_n a_n=1$, and let $\nu: \mathcal{P}(\omega) \to \overline{\mathbb{R}}$ be the submeasure defined by 
\begin{displaymath}
\forall A\subseteq \omega, \quad 
\nu(A):=
\begin{cases}
    \sum_{n \in A}a_n \,\,& \text{ if } A \text{ is finite;}\\
    1+\sum_{n \in A}a_n \,\,\,\,& \text{ if } A \text{ is infinite.}\\
\end{cases}
\end{displaymath}
In particular, $\Sigma=\overline{\Sigma}=\mathcal{P}(\omega)$ and $\nu=\nu^\star$. We get by Proposition \ref{prop:negativesimpleexample} that 
$(\overline{\Sigma}, d_{\nu^\star})$ is not complete. 
Now, observe that $\nu(A)=0$ if and only if $A=\emptyset$, hence the Stone space $S$ is isomorphic to the classical Stone--\v{C}ech compactification of $\omega$, that is, the set of principal ultrafilters $P$ and free ultrafilters $F$ on $\omega$, so that 
$S=P \cup F$ (hence, $P$ and $F$ are identified with their images under isomorphism).  
Let us fix $B\subseteq S$. First, suppose that $B\subseteq P=\bigcup_n P_n$ where $P_n:=\phi(\pi(\{n\}))$ is the singleton containing the unique principal ultrafilter at $\pi(\{n\})$. 
Define $J:=\{j \in \omega: P_j\subseteq B\}$, so that $B=\bigcup_{j \in J} P_{j}$ and $(\{j\}: j \in J) \in \Delta(B)$. It follows that $\hat{\nu}(B)= \sum_{j \in J} \nu(\{j\})=\sum_{j \in J} a_{j}$, hence
\begin{equation}\label{eq:examplehatnu1}
\forall B\subseteq S, \quad B\cap F=\emptyset \implies 
0\le \hat{\nu}(B) \le 1.
\end{equation}

In the opposite case that $B\cap F\neq \emptyset$, i.e., $B$ contains a free ultrafilter $x$. Pick an arbitrary sequence $(A_n: n \in \omega) \in \Delta(B)$. Then $x \in B\subseteq \bigcup_n \phi(\pi(A_n))$, so we can pick $m \in \omega$ such that $x \in \phi(\pi(A_{m}))$. Equivalently, this means that $\pi(A_{m}) \in x$, which is possible only if $A_{m}$ is infinite since $x\notin P$. Considering that $\nu(A_{m})> 1$, it follows that $\hat{\nu}(B)\ge \hat{\nu}(\{x\}) \ge 1$. Together with the monotonicity of $\hat{\nu}$ and the fact by Claim \ref{claim:identitynuhat} that $\hat{\nu}(S)=\nu(\omega)=2$, we obtain 
\begin{equation}\label{eq:examplehatnu2}
\forall B\subseteq S, \quad B\cap F\neq \emptyset \implies 
1\le \hat{\nu}(B) \le 2. 
\end{equation}

At this point, pick $B \subseteq S$ with $\hat{\nu}(B)=0$. It follows by \eqref{eq:examplehatnu1} and \eqref{eq:examplehatnu2} that $B\subseteq P$. By the minimality argument above, we get $B=\emptyset$, that is,  
$
\mathcal{N}_{\hat{\nu}}=\{\emptyset\}. 
$ 
Therefore $\mathcal{N}_{\hat{\nu}}$ is a proper subset of $\mathrm{nwd}(S)$, while $(\overline{\Sigma}, d_{\nu^\star})$ is not complete. 
\end{rmk}

\begin{rmk}\label{rmk:sufficientcompleteness}
Since countable unions of clopen sets are open, one might considering the following streghtening of item \ref{item:E3stone}:
    \begin{enumerate}[label={\rm (\textsc{e}\arabic*)}, itemsep=1mm]
    \setcounter{enumi}{3}
      \item \label{item:E4stone} 
     $\hat{\nu}(\partial U)=0$ for every open set $U \subseteq S$\textup{.}
\end{enumerate}
Thanks to Theorem \ref{thm:stonecharacterization}, the following implication holds: 
\begin{equation}\label{eq:jonclaimedinclusion2}
\ref{item:E4stone} 
\quad \implies \quad 
(\overline{\Sigma}, d_{\nu^\star}) \text{ complete}.
\end{equation}

However, 
the converse implication 
of \eqref{eq:jonclaimedinclusion2} 
fails in general. 
In fact, in the complete pseudometric space given in Example \ref{example:ClCnotsigmafield} we proved that 
there exists an open set $U\subseteq S$ which is a countable union of clopen sets and $U\notin \mathcal{C}$. This implies that $\partial U\neq \emptyset$, so that $\hat{\nu}(\partial U)=1$. Therefore item \ref{item:E4stone} does not hold. 
\end{rmk}

\subsection{Acknowledgments} 
The authors are grateful to 
K.~P.~S.~Bhaskara Rao (Indiana University Northwest, US) and Ilijas Farah (York University, CA) for several helpful communications during the preparation of this manuscript.

\bibliographystyle{amsplain}
\bibliography{idealejon}

\end{document}